\documentclass[]{article}
\usepackage[T1]{fontenc}
\usepackage{graphicx}
\usepackage{tablefootnote}
\usepackage{lscape}
\usepackage[font=small,labelfont=bf]{caption}
\usepackage[usenames]{color}
\usepackage{booktabs} 
\usepackage{anysize}
\marginsize{5cm}{3cm}{3.5cm}{3.5cm}
\usepackage{amsmath}
\usepackage{amsfonts}
\usepackage{amssymb}
\usepackage{listings}
\usepackage{lmodern}
\usepackage[latin2]{inputenc}
\usepackage[english]{babel}
\usepackage{chngpage}
\usepackage{enumerate}
\usepackage{float}
\usepackage{fancyhdr}
\usepackage{amsthm}
\usepackage[title,toc,page]{appendix}

\usepackage{wrapfig}
\usepackage[export]{adjustbox}

\providecommand{\U}[1]{\protect\rule{.1in}{.1in}}
\newtheorem{theorem}{Theorem}
\newtheorem{lemma}{Lemma}

\newtheorem{corollary}{Corollary}

\makeatletter
\@addtoreset{chapter}{part}
\makeatother
\usepackage[hidelinks]{hyperref}
\usepackage{hypcap}
\usepackage{authblk}

\begin{document}
\title{New production matrices for geometric graphs}

\author[1]{Guillermo Esteban\thanks{Emails:
		\texttt{guillermoesteban@cmail.carleton.ca}}}
\author[2]{Clemens Huemer, Rodrigo I. Silveira\thanks{Emails:
		\texttt{\{clemens.huemer, rodrigo.silveira\}@upc.edu}}}
\affil[1]{Carleton University, Ottawa, Canada}
\affil[2]{Universitat Polit\`ecnica de Catalunya, Barcelona, Spain}

\date{}
\maketitle

\begin{abstract}
	We use production matrices to count several classes of geometric graphs. We present novel production matrices for non-crossing partitions, connected geometric graphs, and $ k $-angulations, which provide another way of counting the number of such objects. Counting geometric graphs is then equivalent to calculating the powers of a production matrix. Applying the technique of Riordan Arrays to these production matrices, we establish new formulas for the numbers of geometric graphs as well as combinatorial identities derived from the production matrices. Further, we obtain the characteristic polynomial and the eigenvectors of such production matrices.
\end{abstract}

\section{Introduction}

This work is devoted to the prominent problem of counting geometric graphs. A \emph{geometric graph} on a finite set of points $ \mathcal{S}\subseteq \mathbb{R}^2 $ is a graph with vertex set $ \mathcal{S} $ whose edges are straight-line segments with endpoints in $ \mathcal{S} $. It is called \emph{plane} if no two edges intersect except at common endpoints. We focus on plane graphs on a set $ \mathcal{S} $ of $ n $ points in convex position, which will be labelled $ \{p_{1},...,p_{n}\} $ in counter-clockwise order. Numerous geometric graph classes exist, such as triangulations (i.e., plane graphs all whose faces, except possibly the exterior one, are triangles), connected graphs, or spanning trees. A fundamental problem is to determine the number of  graphs, for each class, as a function of $ n $. Already in 1753, Euler and Segner determined the number of triangulations. These numbers are the well-known Catalan numbers. For many other classes of plane graphs, such as trees, forests, dissections, non-crossing partitions, connected graphs, and geometric graphs, Flajolet and Noy \textcolor{blue}{\cite{Noy}} gave formulas for their numbers. Some of such formulas were also obtained earlier, see for instance \textcolor{blue}{\cite{Kreweras}}, where non-crossing partitions were discussed for the first time, \textcolor{blue}{\cite{Dulucq}}, where trees were first enumerated or \textcolor{blue}{\cite{Comtet}}, where results for dissection were summarized.

\vspace{3.5mm}

The subject of counting plane geometric graphs for points in convex position has been studied intensively. There exist different methods that provide formulas for the number of graphs in a given graph class with $ n $ vertices. For an account of classical results, we refer to Comtet's book \textcolor{blue}{\cite{Comtet}}. Furthermore, Flajolet and Noy \textcolor{blue}{\cite{Noy}} used tools from analytic combinatorics to obtain other formulas for some graph classes unifying specifications of combinatorial structures with generating functions. Also, when analyzing more complex parameters, such as extremal parameters \textcolor{blue}{\cite{Asymp}}, there is also a need for other methods like singularity analysis, first and second moment method or iterated functions.

Throughout this paper we make use of another way of enumerating plane graphs on point sets in convex position: \textit{production matrices}. This method, in combination with \emph{Riordan Arrays} \textcolor{blue}{\cite{Shapiro}}, has been used recently to count triangulations, spanning trees and geometric graphs \textcolor{blue}{\cite{Clemens2, Geometric}}. Our objective is to use this method to obtain new formulas to count these kind of graphs as well as connected graphs and $ k $-angulations.

\vspace{3.5mm}

Recently, it was shown how different classes of geometric graphs can be counted by using an $ n \times n $ matrix $ A_{n} $, called production matrix, associated to each graph class \textcolor{blue}{\cite{Clemens2, Geometric, Clemens3}}. The number of these graphs for a fixed number of vertices is given by (a column of) a power of $ A_{n} $, while their asymptotic number, as $ n $ tends to infinity, is governed by the largest eigenvalue of $ A_{n} $. To derive a production matrix for a certain graph class, first the graphs with $ i \leq n $ vertices are partitioned according to the \emph{degree} of a specified root vertex (several definitions of degree are possible). Then, each part is counted in the entries of an $ n $-element vector $ v^i $ and hence, the sum of the vector elements gives the number of geometric graphs for $ i $ vertices. A production matrix $ A_n $ is a matrix satisfying $ v^{i+1} = A_n v^i $. Furthermore, an initial vector $ v^c $ contains the count for a constant number of vertices (usually $ (1,0,...,0)^{\top} $ and $ c=1 $, but for the class of geometric graphs the initial vector will be $ (2,0,...,0)^{\top} $, in this case we start with $ c=2 $ vertices). It then follows that $ v^{i+1} = A^{i+1-c}_n v^c $ for $ i+1 \geq c $. The method above to define $ A_{n} $ based on a vertex degree, implicitly arranges graphs into a tree structure known as \emph{generating tree} \textcolor{blue}{\cite{Chung}}. In this work we will use different definitions of vertex degree than in \textcolor{blue}{\cite{Clemens2, Geometric}}, which results in different generating trees and different production matrices. We then will use Riordan Arrays to analyze the obtained production matrices giving formulas for the number of graphs, characteristic polynomials and eigenvectors.

\subsection{Previous results}

The concept of production matrix was first introduced and studied by Deutsch et al. \textcolor{blue}{\cite{Combinatorial1}}. Also, these matrices were defined to be infinite and in a triangular form \textcolor{blue}{\cite{Combinatorial2}}. One of the results obtained in \textcolor{blue}{\cite{Combinatorial1}} was a production matrix for Catalan structures, including triangulations, shown in Table \ref{fig:11}(a), and also obtained by Merlini and Verri \textcolor{blue}{\cite{Combinatorial2}}. Recently, Huemer et al. \textcolor{blue}{\cite{Clemens2, Geometric, Clemens3}} obtained production matrices for other graph classes, and computed formulas for the characteristic polynomials of those matrices, as well as formulas for the entries of vectors that enumerate the number of graphs. In this section, we summarize the main previous results. In Table \ref{fig:11}, examples of $ n \times n $ production matrices from \textcolor{blue}{\cite{Clemens2}} for different graph classes for $ n = 6 $ are presented. Matrices (a)-(e) are for graphs with at most $ n $ points and matrix $ (f) $ is for paths with at most $ \frac{n}{2} $ points.

\vspace{3.5mm}

\begin{center}
	\includegraphics[scale=0.9]{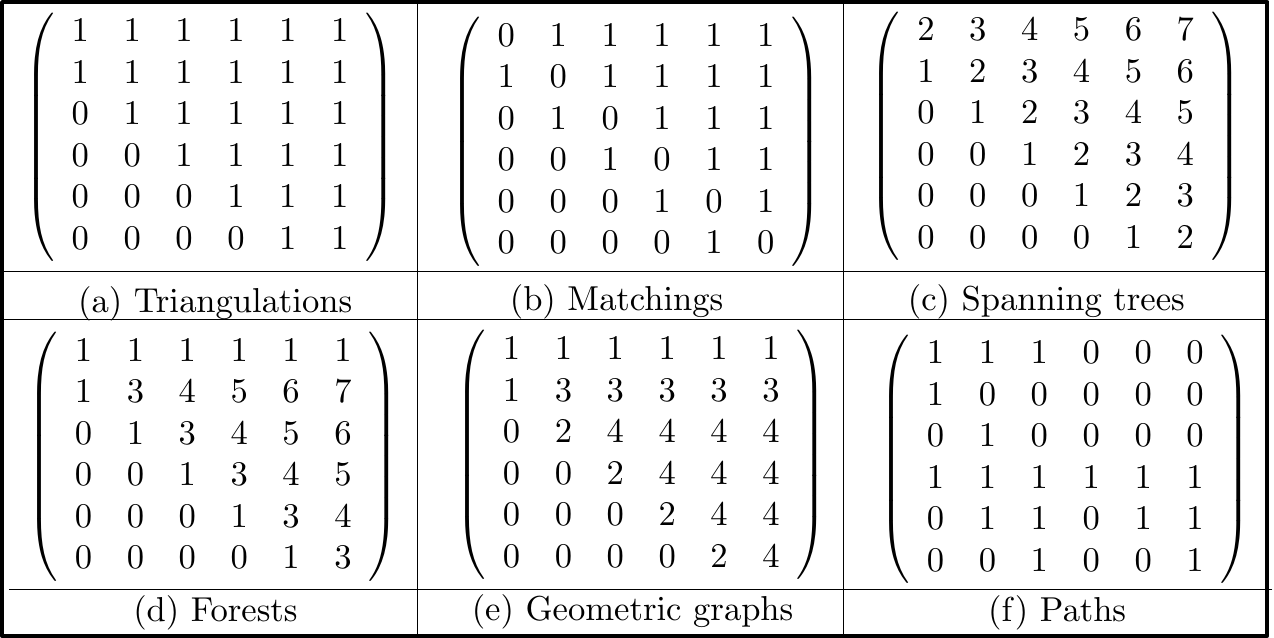}
	\captionof{table}{Production matrices for six different graph classes \textcolor{blue}{\cite{Clemens2}}.}
	\label{fig:11}
\end{center}

\bigskip

In \textcolor{blue}{\cite{Geometric}}, formulas for the entries of the vector $ v^{n} $ were also given. In addition, the characteristic polynomials of the $ n \times n $ production matrices given in Table \ref{fig:11} are given in \textcolor{blue}{\cite{Clemens2}}.

\vspace{3.5mm}

Many of the results obtained in \textcolor{blue}{\cite{Clemens2, Geometric, Clemens3}} are related to Catalan numbers \textcolor{blue}{\cite{Catalan1,Catalan2}} and other well-known sequences of numbers. For instance, the Ballot number $ B_{n,k} $ appears in the formula of the vector for the number of triangulations and also in the one for matchings, due to its relation with Catalan numbers (note that $ \sum_{k}B_{n,k} $ is a Catalan number). Also, a relation between Fibonacci numbers and the class of graphs of forests appears in \textcolor{blue}{\cite{Clemens2}}.

\subsection{Our results}

In this section we summarize the different results presented in the paper. More results in several directions are also given in Esteban's Thesis \textcolor{blue}{\cite{tesis}}.

\vspace{3.5mm}

\textbf{Production matrices}\smallskip

This work is an extension of previous results \textcolor{blue}{\cite{Clemens2, Geometric}} in several directions. We present new production matrices that count $ k $-angulations, geometric graphs, connected geometric graphs, non-crossing partitions, forests, and forests of paths. In the case of geometric graphs and non-crossing partitions production matrices were already known, but we define new production matrices for these two graph classes. We summarize the production matrices we obtained for each graph class studied in Tables \ref{fig:12} and \ref{fig:14}.

\begin{center}
	\includegraphics[scale=0.88]{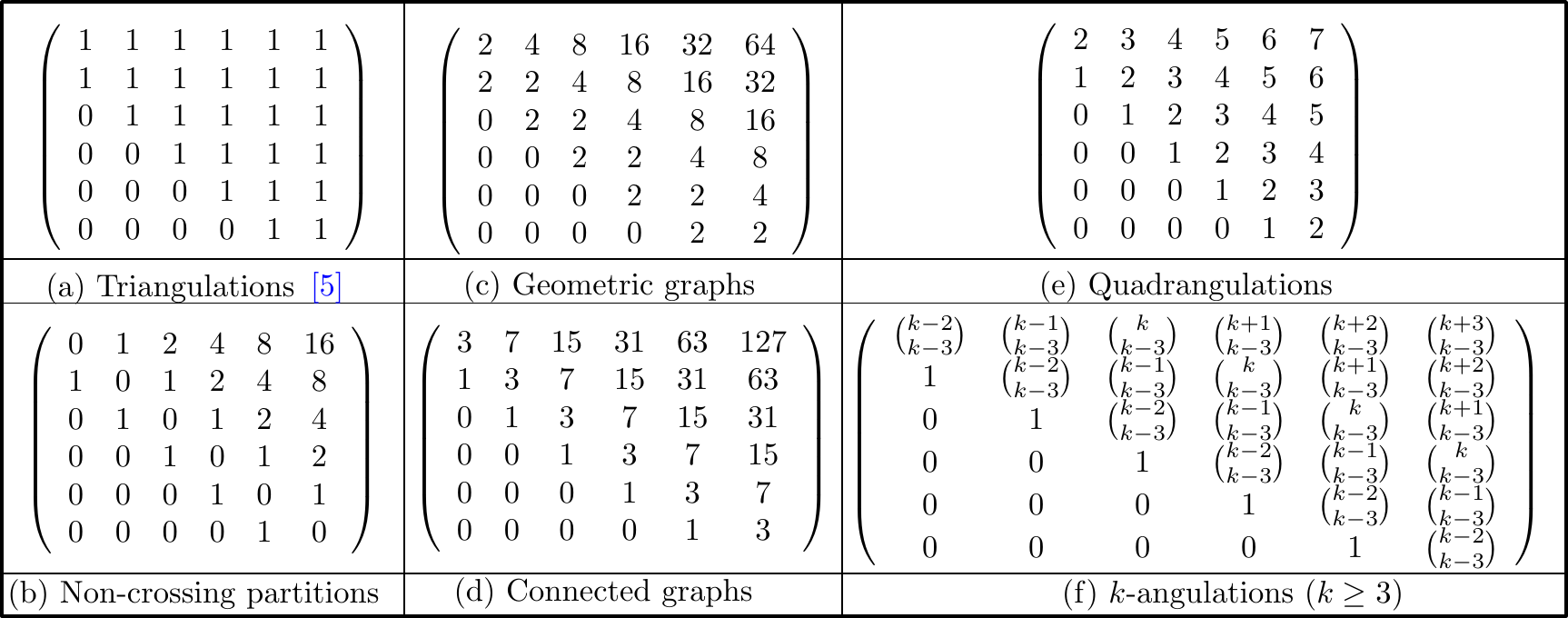}
	\captionof{table}{Production matrices for six graph classes.}
	\label{fig:12}
\end{center}

An important concept that we have to take into account when defining a production matrix is the degree of the root vertex of a graph. While in previous works \textcolor{blue}{\cite{Clemens2, Geometric}}, for most matrices the degree of a vertex was defined as the number of incident edges, in this work we use different definitions of vertex degree. Matrix (a) in Table \ref{fig:12} \textcolor{blue}{\cite{Combinatorial1, Combinatorial2}} is used to count triangulations and other Catalan structures like non-crossing partitions and perfect matchings. In Section 3, we generalize matrix (a) to quadrangulations and, further, to $ k $-angulations, see matrices (e) and (f) in Table \ref{fig:12}. Due to the fact that there exists a bijection between quadrangulations and spanning trees \textcolor{blue}{\cite{Panholzer}}, we know that matrix (e) is also a production matrix for spanning trees, obtained using a different approach than the one applied in \textcolor{blue}{\cite{Geometric}}. For the graph class of $ k $-angulations we use as degree definition the number of edges incident to the root vertex, minus $ 2 $. In Sections 4 and 5 we study production matrices for geometric and connected geometric graphs, see matrices (c) and (d). For these graph classes, the degree of the root vertex is defined in a different way, based on visibility: the degree of the root vertex $ p_{n} $ is the number of visible vertices from a vertex $ p_{n+1} $ inserted between $ p_{1} $ and $ p_{n} $ in convex position, minus $ 2 $. Two vertices are \textit{visible} if the line segment connecting them does not intersect the interior of any edge of the graph. An example to illustrate the concept of \textit{visibility degree} is given in Figure \ref{fig:10}. We observe that in the right graph, the degree of the root vertex $ p_{12} $, for visibility degree and for the number of incident edges, coincides, and it is $ 3 $. However, in the left graph, with the degree definition of \textcolor{blue}{\cite{Geometric}}, the root vertex $ p_{12} $ would have degree $ 0 $, while with the new definition, $ p_{12} $ has visibility degree $ 3 $.

\begin{center}
	\includegraphics[scale=0.73]{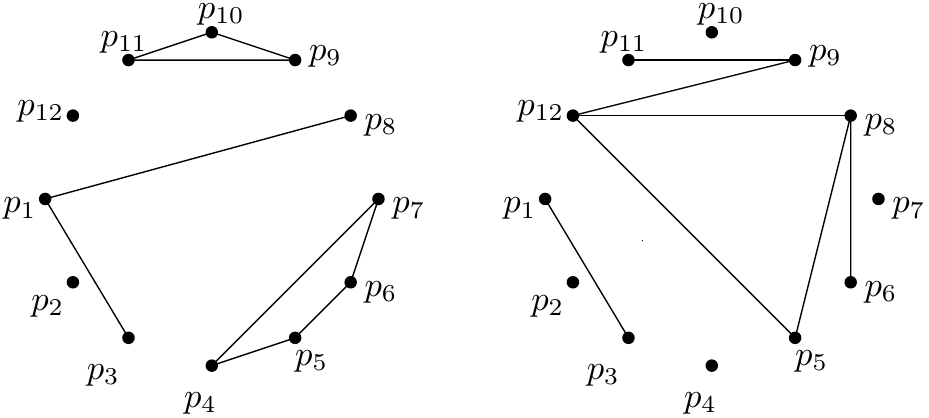}
	\captionof{figure}{Geometric graphs with $ n=12 $ vertices.}
	\label{fig:10}
\end{center}

We point out that matrix (c) in Table \ref{fig:12} is a different matrix for geometric graphs than the one obtained in \textcolor{blue}{\cite{Geometric}}, see Table \ref{fig:11}(e). As for matrix (d), this is the first time that a production matrix for connected geometric graphs is derived, to the best of our knowledge. In Section $ 5.3 $ we present another interesting result: a production matrix that relates the class of geometric graphs to other classes of graphs: connected geometric graphs, spanning trees and paths. The entries of such a matrix are calculated with the sum of the numbers of connected graphs and it gives the number of geometric graphs of the same graph class. For these graph classes, the root vertex degree is based on yet another definition of visibility, called \textit{isolation degree}: the degree of the root vertex is the number of isolated visible vertices from a vertex $ p_{n+1} $ inserted between $ p_{1} $ and $ p_{n} $ in convex position. An example to illustrate the concept of isolation degree is given in Figure \ref{fig:13}. With the degree definition of \textcolor{blue}{\cite{Geometric}}, the root vertex $ p_{12} $ would have degree $ 0 $, while with the new definition, $ p_{12} $ has isolation degree $ 2 $.

\begin{center}
	\includegraphics[scale=0.73]{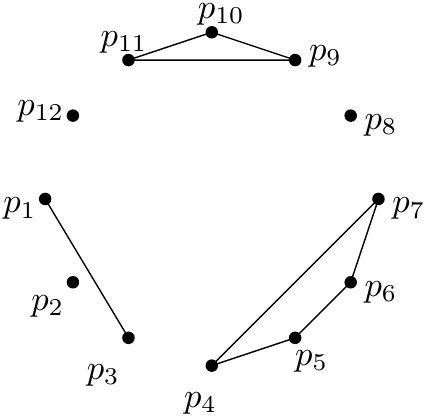}
	\captionof{figure}{A geometric graph (non-crossing partition) with $ n=12 $ vertices and two isolated vertices $ \{p_8, p_{12}\} $. The isolation degree of the root vertex $ g_{12} $ is $ 2 $.}
	\label{fig:13}
\end{center}

Table \ref{fig:14} shows production matrices for three graph classes. The production matrix in Table \ref{fig:14}(a) for geometric graphs is different from the previous ones, i.e., compare to Table \ref{fig:11}(e) and Table \ref{fig:12}(c), and the entries of the first row are calculated with the sum of connected graphs given as Sequence $ A007297 $ in The On-Line Encyclopedia of Integer Sequences (OEIS) \textcolor{blue}{\cite{OEIS}}. In a similar manner, for plane forests in Table \ref{fig:14}(b) (see Sequence $ A054727 $ in OEIS), we find the entries of the first row of the production matrix in OEIS as Sequence $ A307678 $ and they are given by the sum of the number of spanning trees. In the case of spanning forests of paths in Table \ref{fig:14}(c), the entries of the first row of the production matrix correspond to Sequence $ A006234 $ in OEIS and we can calculate them with the sum of spanning paths given as Sequence $ A001792 $ in OEIS.

\begin{center}
	\includegraphics[scale=0.88]{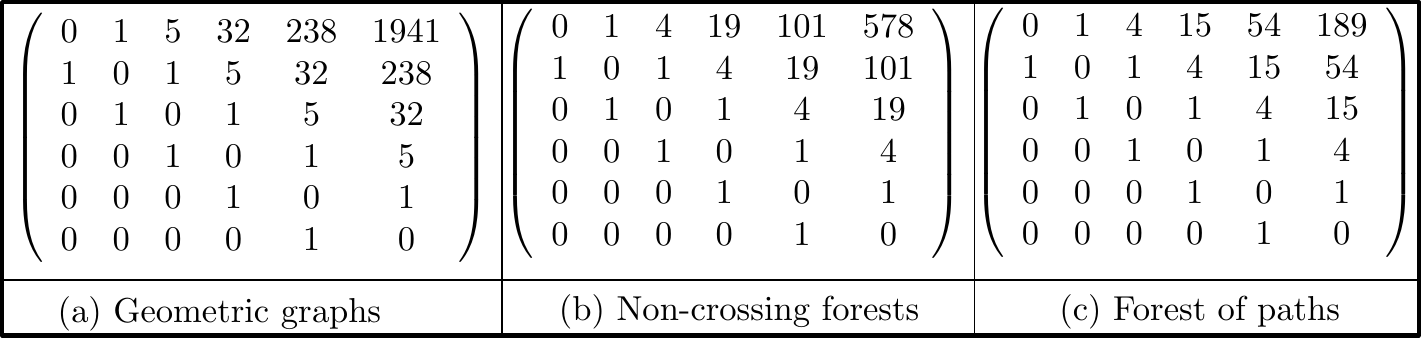}
	\captionof{table}{Production matrices for three graph classes, for $ n = 6 $, obtained by exploiting the relation between geometric graphs and connected graphs.}
	\label{fig:14}
\end{center}\smallskip

Finally, in Section $ 6 $ we study non-crossing partitions. To define the new production matrix, we use the same degree definition as in Section $ 5.3 $: the number of isolated visible vertices from a vertex $ p_{n+1} $ inserted in convex position between $ p_{n} $ and $ p_{1} $ (see Figure \ref{fig:13}). We point out that production matrix (b) in Table \ref{fig:12} provides an alternative production matrix for Catalan structures. This leads to a partition formula of Catalan numbers $ C_n $ that we have not found in the literature, see Theorem \ref{thm:2}, similar to the well-known formula $ C_n = \sum_k B_{n,k} $ with Ballot numbers $ B_{n,k} $.

\vspace{3.5mm}

\textbf{Vectors that count graphs}
\smallskip

In addition to devising production matrices, we deduce formulas for vectors counting graphs with a given root degree. The results for the entries of those vectors for the different graph classes studied are stated in Table \ref{fig:1}. For $ k $-angulations, the formula counts the number of $ k $-angulations with $ r $ $ k $-gons, on $ n $ points, and root vertex degree $ j-1 $, for $ j=1,\ldots,r $. For non-crossing partitions, the formula counts the number of non-crossing partitions with $ n $ points, and root vertex degree $ j-1 $, for $ j=1,\ldots,n+1 $. The other formulas enumerate the graphs with $ n $ vertices and root vertex degree $ j-1 $, for $ j = 1,\ldots,n-1 $.

\begin{center}
	\includegraphics[scale=0.73]{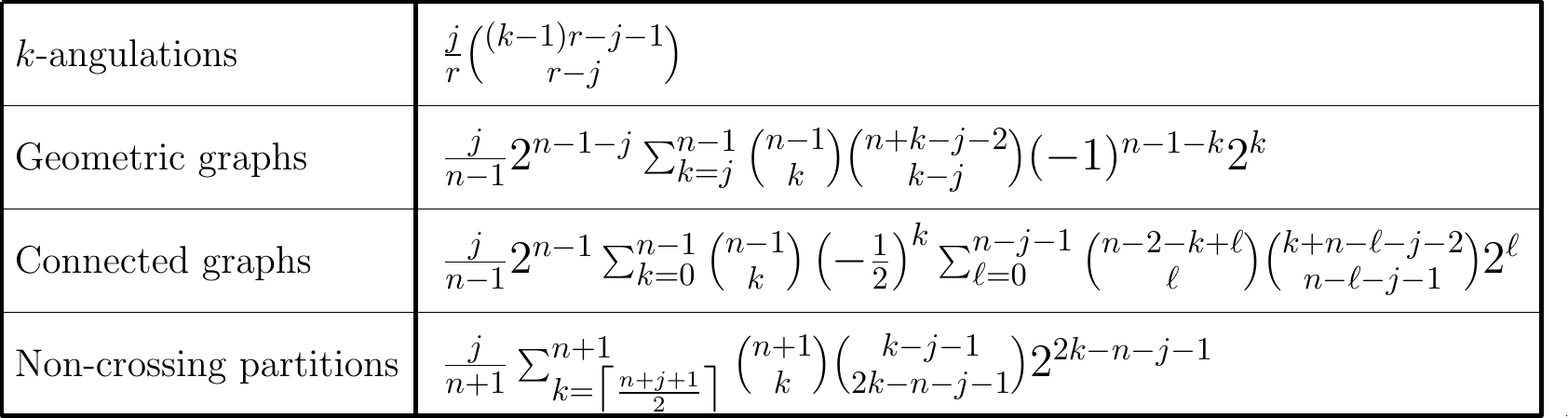}
	\captionof{table}{Entries for vectors that enumerate the number of graphs for four graph classes.}
	\label{fig:1}
\end{center}

\textbf{Characteristic polynomials}
\smallskip

We also study other properties that have combinatorial implications. We obtain the characteristic polynomials of the matrices, and a characterization of the eigenvectors associated to them. The recursive formulas obtained for the characteristic polynomials of the production matrices of the graph classes stated in Table \ref{fig:12} are summarized in Table \ref{fig:2}.

\begin{center}
	\includegraphics[scale=0.65]{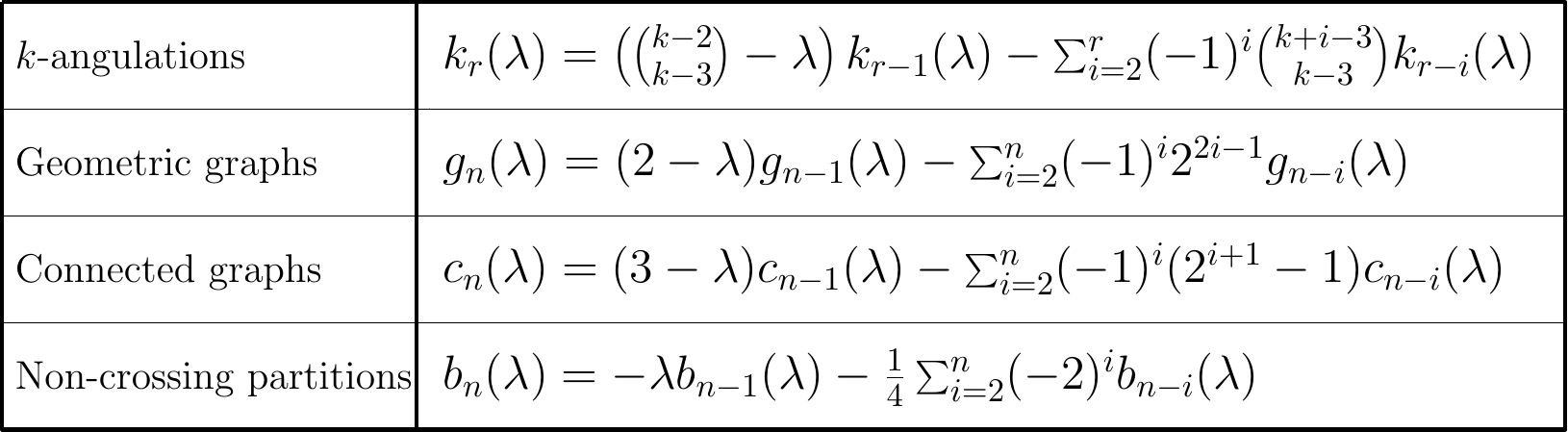}
	\captionof{table}{Recursive formulas for the characteristic polynomials for the production matrices for four graph classes.}
	\label{fig:2}
\end{center}

Finally, the solutions for the recursive formulas of the characteristic polynomials stated in Table \ref{fig:2} are given in Table \ref{fig:3}.

\begin{center}
	\includegraphics[scale=0.7]{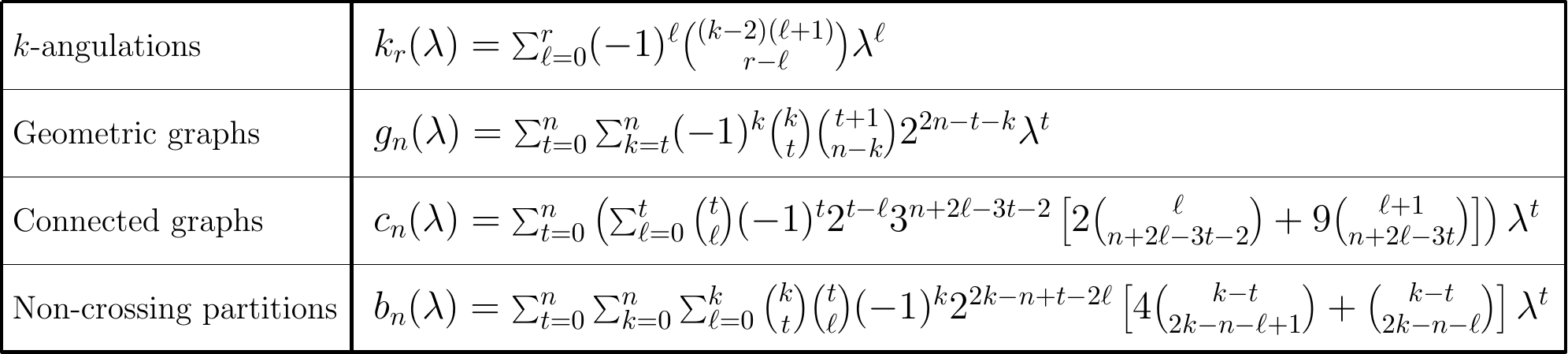}
	\captionof{table}{Characteristic polynomials for the production matrices for four graph classes.}
	\label{fig:3}
\end{center}

\section{Preliminaries}

For the sake of clarity, the concepts explained in this section are divided into two sections. In Section $ 2.1 $, we introduce the tools needed in the following chapters. It includes both the main definitions that appear in other sections, and the main theorems we use in this work. Finally, Section $ 2.2 $ is devoted to find the characteristic polynomial and the form of an eigenvector of Hessenberg-Toeplitz matrices, something that we will use in the following sections to prove formulas in Table \ref{fig:2}.

\subsection{Analytic tools}

Throughout this work, we use the concept of generating trees and some theorems related to Riordan Arrays. A generating tree is a rooted labelled tree with the property that if $ v_1 $ and $ v_2 $ are any two nodes with the same label then, for each label $ \ell $, $ v_1 $ and $ v_2 $ have exactly the same number of children with label $ \ell $. A more detailed explanation about this kind of trees can be found in the work of Merlini and Verri \textcolor{blue}{\cite{Combinatorial2}}. Moreover, for graphs on point sets in convex position, generating trees for triangulations \textcolor{blue}{\cite{Triang}} and spanning trees \textcolor{blue}{\cite{Span}} have been obtained. Generating trees have been studied in several contexts, most notably for the \emph{ECO method} \textcolor{blue}{\cite{ECO}}, and for obtaining matrix representations of diverse combinatorial objects \textcolor{blue}{\cite{Combinatorial1, Combinatorial2}}. Once these matrices are obtained, one can follow the approach used by Merlini and Verri \textcolor{blue}{\cite{Combinatorial2}}, using the theory of Riordan Arrays to analyze them.

\vspace{3.5mm}

A Riordan Array is an infinite lower triangular array $ \{d_{n, j}\}_{n,j \in \mathbb{N}} $, defined by a pair of generating functions $ (d(t), h(t)) $, such that the generic element $ d_{n, j} $ is the $ n $-th coefficient in the series $ d(t)(th(t))^j $, i.e.,
\begin{equation}
	\label{eq:4}
	d_{n, j} = [t^n]d(t)(th(t))^j, \qquad n,j \geq 0.
\end{equation}

From this definition, $ d_{n, j} = 0 $ for $ j > n $. We always assume that $ d(0) \neq 0 $. If we also have $ h(0) \neq 0 $, then the Riordan Array is said to be \textit{proper}. Let $ D = (d(t),h(t)) $ be a proper Riordan Array, let $ A = \{a_{i}\}_{i \in \mathbb{N}} $ be its A-sequence, and let $ Z = \{z_{i}\}_{i \in \mathbb{N}} $ be its Z-sequence (the Z-sequence characterizes row $ 0 $ of the Riordan Array while the A-sequence characterizes all the other rows \textcolor{blue}{\cite{Combinatorial2}}), then:
\begin{equation}
\label{eq:13}
	h(t) = A(th(t)) \qquad \qquad d(t) = \frac{d(0)}{1-tZ(th(t))},
\end{equation}

where $ A(t) $ and $ Z(t) $ are the generating functions of the A-sequence and Z-sequence, respectively.

\bigskip

One theorem that we will use, especially when determining the entries of the vector $ v^{i} $ that counts the number of elements of a certain graph class, as well as for the characteristic polynomials of their production matrices is Theorem \ref{thm:3}. Theorem \ref{thm:5} connects the concept of proper Riordan Arrays with the one of transfer matrices used by West \textcolor{blue}{\cite{West}}, pointing out the power of the Riordan Array approach.

\begin{theorem}
	\label{thm:3}
	\textbf{(Merlini and Verri \textcolor{blue}{\cite{Combinatorial2}})}
	\smallskip
	
	Let $ D = (d(t), h(t)) $ be a proper Riordan Array, and let $ A = \{a_{i}\}_{i \in \mathbb{N}} $ be its A-sequence. Then, if $ d(t) = 1 $ we have:
	
	\begin{equation*}
		d_{n,j} = \frac{j}{n}[t^{n-j}]A(t)^{n},
	\end{equation*}
	
	and if $ d(t) = h(t) $ we have:
	
	\begin{equation*}
		d_{n,j} = \frac{j+1}{n+1}[t^{n-j}]A(t)^{n+1}.
	\end{equation*}
	
\end{theorem}

\begin{theorem}
	\label{thm:5}
	\textbf{(Merlini and Verri \textcolor{blue}{\cite{Combinatorial2}})}
	\smallskip
	
	Let $ D = \{d_{n, j}\}_{n,j \in \mathbb{N}} $ be a proper Riordan Array defined by the triple $ (d(0), A, Z) $ and let $ T $ and $ \bar{d}_0 $ be an infinite matrix and an infinite vector defined as follows:
	\begin{equation*}
		T = \begin{pmatrix}
			z_{1} & z_{2} & z_{3} & z_4 & \cdots \\
			a_{1} & a_{2} & a_{3} & a_4 & \cdots \\
			0 & a_{1} & a_{2} & a_3 & \cdots \\
			0 & 0 & a_{1} & a_2 & \cdots \\
			\vdots & \vdots & \vdots & \vdots & \ddots \\
			\end{pmatrix}, \qquad \bar{d}_0 = \begin{pmatrix}
			d(0) \\
			0 \\
			0 \\
			0 \\
			\vdots \\
			\end{pmatrix}.
	\end{equation*}
	Then, the generic element $ d_{n,j} $ of the Riordan Array, $ n > 0 $, is the $ j $-th element of the vector $ T^n\bar{d}_0 $.
\end{theorem}

It follows that the entries $ d_{n,j} $ correspond to the entries of the vector $ v^{n} $ when $ T $ is a production matrix and $ d(0) $ is the non-zero entry of the initial vector $ v^c $.\smallskip

\subsection{Hessenberg-Toeplitz matrices}
\smallskip

All the production matrices we obtain are Upper Hessenberg-Toeplitz, i.e., matrices of the form:
\begin{equation*}
	A_{n} = \begin{pmatrix}
		a_{0} & a_{1} & a_{2} & \cdots & a_{n-1} \\
		a_{-1} & a_{0} & a_{1} & \cdots & a_{n-2} \\
		0 & a_{-1} & a_{0} & \cdots & a_{n-3} \\
		\vdots & \vdots & \vdots & \ddots & \vdots \\
		0 & 0 & 0 & \cdots & a_{0} \\
	\end{pmatrix}
\end{equation*}

A recursive formula for the determinant of Hessenberg-Toeplitz matrices $ A_{n} $ has already been given in \textcolor{blue}{\cite{Dale}}. Hence, we can obtain easily the following result for the characteristic polynomial of $ A_{n} $ from \textcolor{blue}{\cite[Theorem $ 4.20 $]{Dale}}.

\vspace{3.5mm}

\begin{theorem}
	\label{thm:13}
	\textbf{(Dale and Vein \textcolor{blue}{\cite{Dale}})}
	\smallskip
	
	The characteristic polynomial $ d_{n}(\lambda) $ of a Hessenberg-Toeplitz matrix $ A_{n} $, for $ n \geq 1 $, satisfies the recurrence relation
	\begin{equation*}
		d_{n}(\lambda) = (a_{0}-\lambda)d_{n-1}(\lambda) + \sum_{i=2}^{n}(-1)^{i+1}a_{i-1}a_{-1}^{i-1}d_{n-i}(\lambda),
	\end{equation*}
	where $ d_0(\lambda) = 1 $.
\end{theorem}

\bigskip

In the following, we prove a formula for the entries of an eigenvector associated to an eigenvalue $ \lambda $ of a Hessenberg-Toeplitz matrix that will be useful to find formulas for the eigenvectors of the production matrices for the classes of graphs we study.

\begin{theorem}
	\label{thm:14}
	Let {\small $ x=(x_{n-1},x_{n-2},\ldots,x_{0})^{\top} $} be an eigenvector associated to an eigenvalue $ \lambda $ of a Hessenberg-Toeplitz matrix $ A_{n} $, for $ n \geq 2 $. Then, the entries of the vector $ x $ are of the form:
	
	\begin{equation*}
		x_{i} = \left(\frac{-1}{a_{-1}}\right)^{i}d_{i}(\lambda)x_{0}
	\end{equation*}
	
	$ \forall i=1,\ldots, n-1 $, where $ d_{i}(\lambda) $ is the characteristic polynomial of $ A_{i} $.
\end{theorem}

\begin{proof}
	We will prove the theorem by induction on the $ i $-th entry of the vector and taking into account that $ A_{n}x = \lambda x $, where $ \lambda $ is the eigenvalue associated to an eigenvector $ x $.
	
	\vspace{3.5mm}
	
	\textbf{Inductive hypothesis}: The base case for $ i=1 $ is easily verified, so we will suppose that 
	
	\vspace{3.5mm}
	\begin{equation*}
		x_{i} = \left(\frac{-1}{a_{-1}}\right)^{i}d_{i}(\lambda)x_{0}, \text{ \qquad for all } i < n-2.
	\end{equation*}
	
	\textbf{Inductive step}: if we multiply the second row of the production matrix $ A_n $ by the eigenvector $ x $ and we make it equal it to $ \lambda x_{n-2} $, we get:
	\begin{equation*}
		a_{-1}x_{n-1} = (\lambda - a_{0})x_{n-2} -a_{1}x_{n-3} -a_{2}x_{n-4} -\ldots -a_{n-2}x_{0}.\\
	\end{equation*}
	
	Then, by induction hypothesis:
	\begin{multline*}
		a_{-1}x_{n-1} = \left[(\lambda - a_{0})\frac{(-1)^{n-2}}{a_{-1}^{n-2}}d_{n-2}(\lambda) - a_{1}\frac{(-1)^{n-3}}{a_{-1}^{n-3}}d_{n-3}(\lambda) - \ldots-a_{n-2}\right]x_{0} \Longrightarrow \\
		\Longrightarrow x_{n-1} = \frac{(-1)^{n-1}}{a_{-1}^{n-1}}\left[(a_{0}-\lambda)d_{n-2}(\lambda) + \sum_{i=2}^{n-1} (-1)^{i+1}a_{i-1}a_{-1}^{i-1}d_{n-i-1}(\lambda)\right]x_{0}.
	\end{multline*}
	
	And by Theorem \ref{thm:13}, this is equal to $ \frac{(-1)^{n-1}}{a_{-1}^{n-1}}d_{n-1}(\lambda)x_{0} $. So the last part of the proof consists in developing the first row of the system of equations $ A_{n}x = \lambda x $ given by: 
	\begin{equation*}
		a_{0}x_{n-1} + a_{1}x_{n-2} + \ldots + a_{n-1}x_{0} = \lambda x_{n-1} \ \Longrightarrow \ (a_{0} - \lambda)x_{n-1} + a_{1}x_{n-2} + \ldots + a_{n-1}x_{0} = 0.
	\end{equation*}
	
	Now, if we substitute each $ x_{i} $ by its value, we get that:
	\begin{multline*}
		\left[(a_{0}-\lambda)\frac{(-1)^{n-1}}{a_{-1}^{n-1}}d_{n-1}(\lambda)+a_{1}\frac{(-1)^{n-2}}{a_{-1}^{n-2}}d_{n-2}(\lambda)+a_{2}\frac{(-1)^{n-3}}{a_{-1}^{n-3}}d_{n-3}(\lambda) +\ldots+a_{n-1}\right]x_{0} = \\
		= \frac{(-1)^{n-1}}{a_{-1}^{n-1}}\left[(a_{0}-\lambda)d_{n-1}(\lambda) +\sum_{i=2}^{n} (-1)^{i+1}a_{i-1}a_{-1}^{i-1}d_{n-i}(\lambda)\right]x_{0} = \frac{(-1)^{n-1}}{a_{-1}^{n-1}}d_{n}(\lambda)x_{0}.
	\end{multline*} 
	
	And this is equal to $ 0 $ if $ d_{n}(\lambda) = 0 $, but this is true because $ d_{n}(\lambda) $ is the characteristic polynomial of $ A_{n} $.
\end{proof}

\section{Counting \texorpdfstring{$ k $-angulations}{}}

In this section we introduce the framework of production matrices and generating trees with our first graph class, the $ k $-angulations. A $ k $-angulation is a 2-connected plane graph in which every internal face is a $ k $-gon. In this section we count the number of $ k $-angulations that consist of exactly $ r $ $ k $-gons. Such a graph has $ n=(k-2)r+2 $ vertices.

For $ k = 3 $ we fall into the well-known case of triangulations. The number of triangulations of a set of $ n + 2 $ points in convex position is the Catalan number $ C_{n} = \binom{2n}{n}\frac{1}{n+1} $.
In general, the number of $ k $-angulations with $ r $ $ k $-gons is given by $ \frac{1}{(k-2)r+1}{\binom{(k-1)r}{r}} $. The way of calculating the entries of the production matrices of $ k $-angulations is similar to the mapping used by Hurtado and Noy \textcolor{blue}{\cite{Triang}} to obtain a generating tree for triangulations. We use as degree of the root vertex the number of its incident edges, minus 2. We substract 2 because the two edges incident to the root vertex that are incident to the outer face are part of every $ k $-angulation and hence, they can be disregarded. Vector $ v^{r} $ counts $ k $-angulations with $ r $ $ k $-gons, and $ (k-2)r+2 $ vertices. Once we have the $ r \times r $ production matrix of $ k $-angulations $ K_{r} $ we apply the methods of Section $ 2.1 $ to get a formula for the number $ v^{r}_{j} $ of $ k $-angulations with $ r $ $ k $-gons and with root vertex of degree $ j-1 $, given by the $ j $-th entry of $ v^{r} $.

\subsection{Production matrix}

Our production matrix is devised by using a surjective mapping of $ k $-angulations with $ (r+1) $ $ k $-gons to $ k $-angulations with $ r $ $ k $-gons. Let $ \sigma $ be a $ k $-angulation with point set $ \{p_{1}, \ldots, p_{(k-2)r+2}\} $ labelled counter-clockwise, where the vertex $ p_{(k-2)r+2} $ is the root vertex and its degree is defined as the number of its incident edges, minus $ 2 $. Then, the entries of $ v^{r} $ will depend on the degree of the root vertex $ p_{(k-2)r+2} $ of $ \sigma $ in the mapping.

\vspace{3.5mm}

The mapping is based on a graph operation called \textit{edge flipping}. For the sake of simplicity, we will explain the case for $ k = 4 $, the remaining cases are analogous. Suppose that we have two quadrilateral faces $ p_1p_2p_3p_4 $ and $ p_1p_4p_5p_6 $ adjacent to edge $ p_{1}p_{4} $ (see Figure \ref{fig:23}(b)). Edge flipping is an operation replacing the diagonal $ p_{1}p_{4} $ with $ p_{3}p_{6} $ (Figure \ref{fig:23}(a)), or with $ p_{2}p_{5} $ (Figure \ref{fig:23}(c)). If a diagonal flip yields multiple edges or loops, then we do not apply it. This operation clearly transforms a quadrangulation into another one.

\begin{center}
	\includegraphics[scale=0.75]{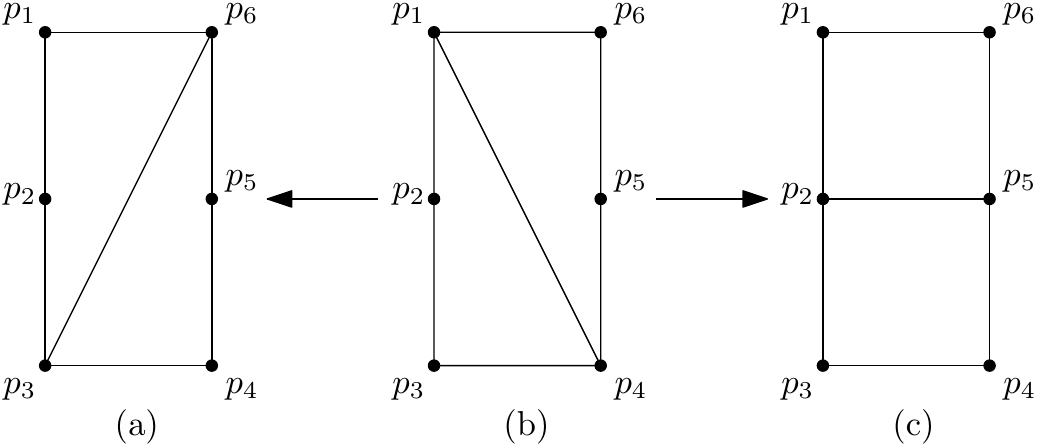}
	\captionof{figure}{The diagonal flip.}
	\label{fig:23}
\end{center}

Once we have defined the tool needed to define the mapping, let us apply it to $ k $-angulations. We define the generating tree of $ k $-angulations by using the following mapping between the set $ \mathfrak{Q}_{r+1} $ of $ k $-angulations with $ (r+1) $ $ k $-gons and $ \mathfrak{Q}_{r} $. Let $ d(p_{i}) $ be the degree of vertex $ p_{i} $ in a $ k $-angulation $ \sigma' \in \mathfrak{Q}_{r+1} $ ignoring the boundary edges. We can obtain a $ k $-angulation $ \sigma \in \mathfrak{Q}_{r} $ from $ \sigma' $ by following this procedure: if $ \sum_{i=3}^{k}d(p_{(k-2)r+i}) = 0 $ (see Figure \ref{fig:8}(c) where $ k=4 $ and vertices $ p_{2r+3} $ and $ p_{2r+4} $ have degree $ 0 $), we just delete vertices $ p_{(k-2)r+3}, \ldots, p_{(k-2)r+k} $. Otherwise, we flip all the edges incident to those vertices so that they become incident to $ p_{(k-2)r+2} $, in order to have no edges incident to $ p_{(k-2)r+3}, \ldots, p_{(k-2)r+k} $; then the $ k $-angulation contains the face with vertices $ \{p_1,p_{(k-2)r+2}, \ldots, p_{(k-2)r+k}\} $ and we can delete this face (see Figure \ref{fig:8}(d)). We call $ \sigma' $ the \textit{child} of $ \sigma $ and $ \sigma $ the \textit{parent} of $ \sigma' $.

\begin{center}
	\includegraphics[scale=0.8]{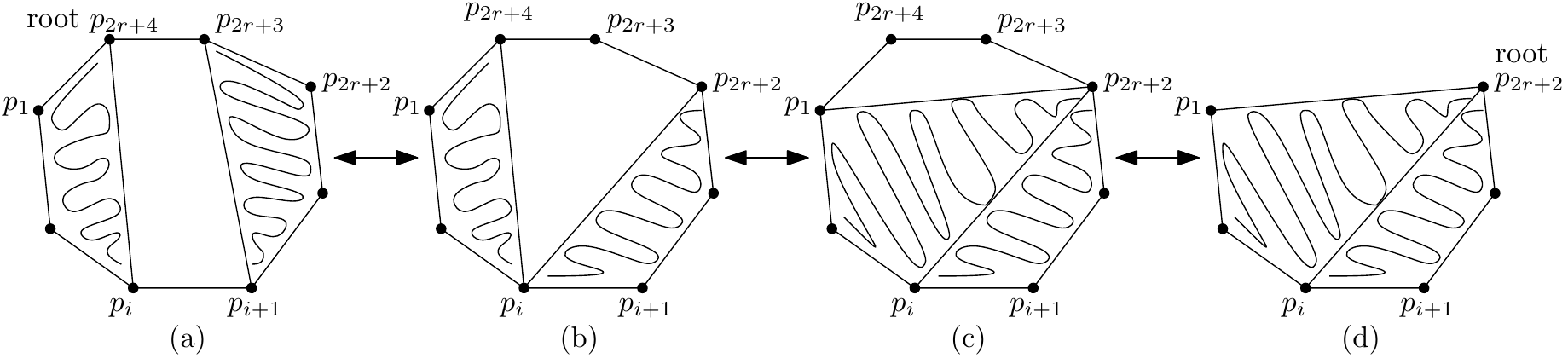}
	\captionof{figure}{Construction of a child $ \sigma' $ (a) of a $ 4 $-angulation $ \sigma $ (d).}
	\label{fig:8}
\end{center}

\bigskip

In the opposite direction, observe that the number of children of $ \sigma $ depends on the vertex degree of $ p_{(k-2)r+2} $, and that a child can always be obtained by adding $ k-2 $ vertices between $ p_{(k-2)r+2} $ and $ p_{1} $ in convex position and flipping a subset of edges incident to $ p_{(k-2)r+2} $, such that they are incident to some of the new vertices (refer again to Figure \ref{fig:8}).

\vspace{3.5mm}

Observe that, by this procedure, each $ k $-angulation is generated exactly once. Once we know how to derive $ \mathfrak{Q}_{r+1} $ from $ \mathfrak{Q}_{r} $, we are able to produce the generating tree of $ k $-angulations, see Figure \ref{fig:5} for the first levels of the generating tree of quadrangulations.

\begin{center}
	\includegraphics[scale=0.7]{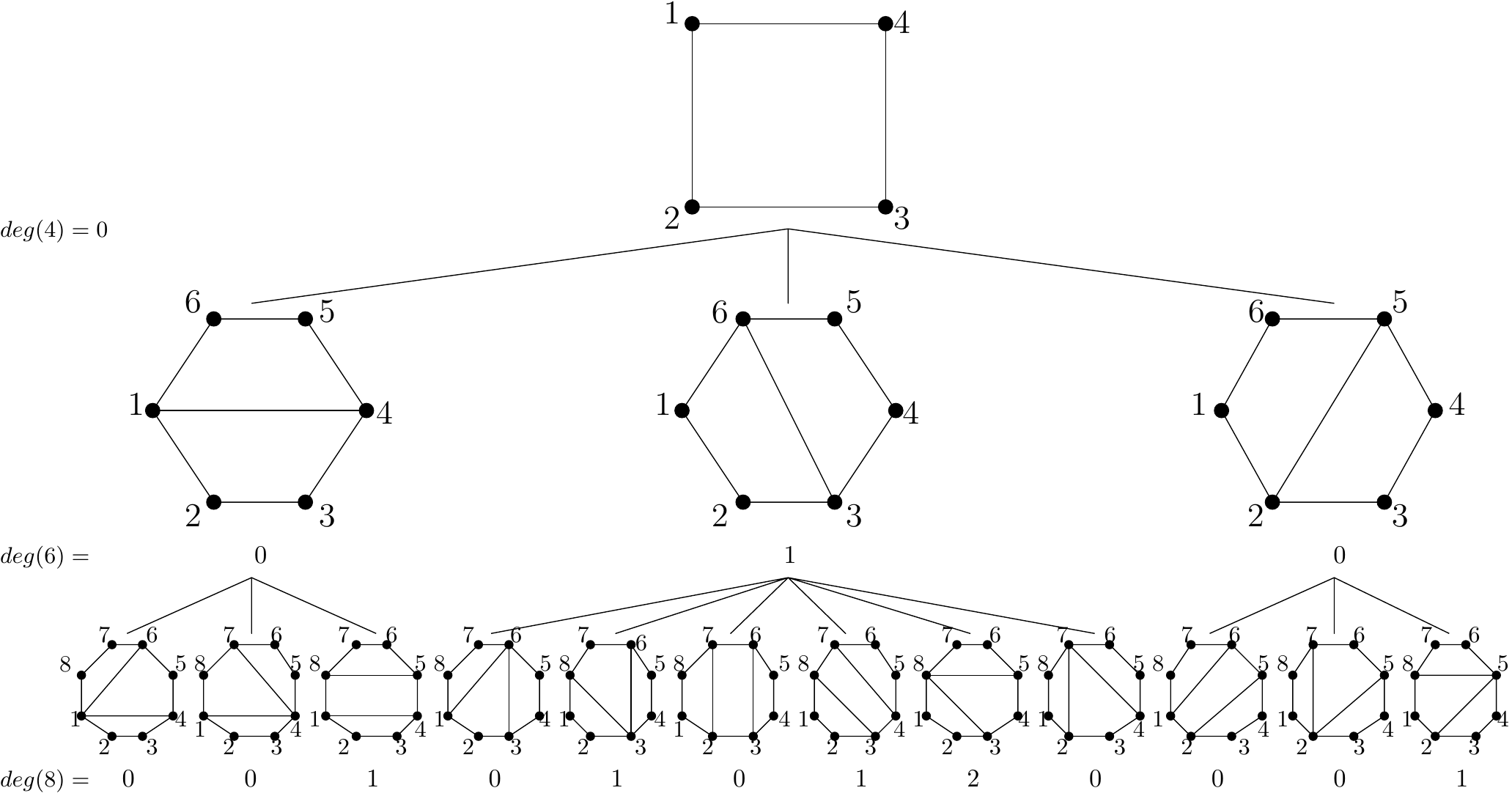}
	\captionof{figure}{First levels of the generating tree of quadrangulations. The degree of the root vertex is given below the quadrangulation.}
	\label{fig:5}
\end{center}

\bigskip

Considering $ v^{i} $ as an $ r $-dimensional vector with $ i \leq r $, we show that $ v^{i+1} $ can be obtained by multiplying $ v^{i} $ by a production matrix $ K_{r} $, that is, $ v^{i+1} = K_{r}v^{i} $.

\begin{theorem}
	The following $ r \times r $ matrix is a production matrix for $ k $-angulations of point sets in convex position.
	
	\begin{equation*}
		K_{r} = \begin{pmatrix}
			\binom{k-2}{k-3} & \binom{k-1}{k-3} & \binom{k}{k-3} & \binom{k+1}{k-3} & \binom{k+2}{k-3} & \cdots & \binom{r+k-3}{k-3} \\
			1 & \binom{k-2}{k-3} & \binom{k-1}{k-3} & \binom{k}{k-3} & \binom{k+1}{k-3} & \cdots & \binom{r+k-4}{k-3} \\
			0 & 1 & \binom{k-2}{k-3} & \binom{k-1}{k-3} & \binom{k}{k-3} & \cdots & \binom{r+k-5}{k-3} \\
			0 & 0 & 1 & \binom{k-2}{k-3} & \binom{k-1}{k-3} & \cdots & \binom{r+k-6}{k-3} \\
			0 & 0 & 0 & 1 & \binom{k-2}{k-3} & \cdots & \binom{r+k-7}{k-3} \\
			\vdots & \vdots & \vdots & \vdots & \vdots & \ddots & \vdots \\
			0 & 0 & 0 & 0 & 0 & \cdots & \binom{k-2}{k-3}
		\end{pmatrix}.
	\end{equation*}	
\end{theorem}

\begin{proof}
	We want to produce the number of graphs with $ r+1 $ $ k $-gons with root vertex of degree $ j-1 $, from $ v^{r} $. An entry of $ v^{r+1} $ is obtained by summing the number of children with a given degree, obtained from the graphs counted in $ v^{r} $, which corresponds to multiplication of a row of $ K_r $ with $ v^r $. Assume we are given a $ k $-angulation $ \sigma \in \mathfrak{Q}_{r} $, and the root vertex $ p_{(k-2)r+2} $ has degree $ t $ (i.e., $ t+2 $ incident edges) in $ \sigma $. First, add $ k-2 $ new vertices $ p_{(k-2)r+3}, \ldots, p_{(k-2)r+k} $ to the $ k $-angulation by creating a new face $ p_{1}, p_{(k-2)r+2}, p_{(k-2)r+3}, $ $ \ldots, p_{(k-2)r+k}p_{1} $. This gives a $ k $-angulation $ \sigma' \in \mathfrak{Q}_{r+1} $. From the definition of the generating tree of $ k $-angulations, and the relation $ v^{i+1} = K_{r}v^{i} $, we can determine the entries of $ K_{r} $ as follows:  
	
	\begin{itemize}
		\item \textbf{First row}. If we want $ d(p_{(k-2)r+k}) = 0 $ in $ \sigma' $, we have the possibility of flippling some of the $ t+1 $ edges incident to $ p_{(k-2)r+2} $ in $ \sigma' $, except the ones joining it with $ p_{(k-2)r+1} $ and $ p_{(k-2)r+3} $, such that they are then incident to $ p_{(k-2)r+3}, \ldots, p_{(k-2)r+k-1} $, i.e., replacing an edge $ p_{(k-2)r+2}p_{i} $ by one of the edges $ p_{(k-2)r+2+j}p_{i+j} $ for all $ 1 \leq j \leq k-3 $, obtaining $ \binom{t+k-2}{k-3} $ $ k $-angulations. This can be seen as the number of ways of distributing $ t+1 $ indistinguishable balls into $ k-2 $ identical boxes.
		
		\item \textbf{Other rows}. The following rows are analogous, shifted by one column every time: if we want $ d(p_{(k-2)r+k}) = m, \ m \leq t+1 $, we have to flip $ m $ edges from $ p_{(k-2)r+2} $ to $ p_{(k-2)r+k} $, and the remaining $ t+1-m $ edges of $ p_{(k-2)r+2} $, we can leave them incident to $ p_{(k-2)r+2} $ or flip a subset of them from $ p_{(k-2)r+2} $ to the vertices $ p_{(k-2)r+3}, \ldots, p_{(k-2)r+k-1} $, obtaining $ \binom{t-m+k-2}{k-3} $ $ k $-angulations, i.e., the number of ways of connecting $ t-m+1 $ edges with $ k-2 $ vertices. This leads us to $ K_{r} $. \qedhere
	\end{itemize}

\end{proof}

\bigskip

As we can see, the elements of each row of $ K_{r} $ correspond to the $ (k-2) $-th diagonal of the Pascal triangle.

\vspace{3.5mm}

Next, we study the number of $ k $-angulations with $ r $ $ k $-gons with root vertex of degree $ j-1 $.

\begin{theorem}
	\label{thm:32}
	The number of $ k $-angulations with $ r $ $k $-gons where the root vertex has degree $ j-1 $ $ \forall j=1,\ldots,r $ is:
	\begin{equation*}
		v_{j}^{r} = \frac{j}{r}\binom{(k-1)r-j-1}{r-j}.
	\end{equation*}
\end{theorem}

\begin{proof}
	
The production matrix $ K_{r} $ has a corresponding Riordan Array. Using notation as in Section $ 2.1 $, its A-sequence is $ \{1, \binom{k-2}{k-3}, \binom{k-1}{k-3}, \binom{k}{k-3}, \binom{k+1}{k-3}, \ldots\} $, with generating function $ A(t) = \frac{1}{(1-t)^{k-2}} $, and its Z-sequence is $ \{\binom{k-2}{k-3}, \binom{k-1}{k-3}, \binom{k}{k-3}, \binom{k+1}{k-3}, \binom{k+2}{k-3}, \ldots\} $, with generating function $ Z(t) = \frac{1}{t(1-t)^{k-2}} - \frac{1}{t} $. We have $ d(0) = v_1^1 = 1 $ and by formula (\ref{eq:13}):

\begin{equation}
	\label{eq:3}
	h(t) = A(th(t)) = \frac{1}{(1-th(t))^{k-2}}.
\end{equation}

We also calculate $ d(t) $ using formula (\ref{eq:13}):

\begin{equation*}
	d(t) = \frac{d(0)}{1-tZ(th(t))} = \frac{1}{1-t\left(\frac{1}{th(t)(1-th(t))^{k-2}} - \frac{1}{th(t)}\right)} = \frac{th(t)(1-th(t))^{k-2}}{th(t)(1-th(t))^{k-2}-t+t(1-th(t))^{k-2}}.
\end{equation*}

We substitute $ th(t) = \omega $ and obtain:

\begin{multline*}
	d(t) = \frac{\omega(1-\omega)^{k-2}}{\omega(1-\omega)^{k-2}-\frac{\omega}{h(t)}+\frac{\omega}{h(t)}(1-\omega)^{k-2}} = \frac{(1-\omega)^{k-2}}{(1-\omega)^{k-2}-\frac{1}{h(t)}+\frac{(1-\omega)^{k-2}}{h(t)}} = \\
	= \frac{h(t)(1-\omega)^{k-2}}{h(t)(1-\omega)^{k-2}-1+(1-\omega)^{k-2}}.
\end{multline*}

Finally, we substitute the value of $ h(t) $ previously calculated in (\ref{eq:3}) and the equation $ th(t) = w $:

\begin{equation*}
	d(t) = \frac{1}{(1-\omega)^{k-2}}.
\end{equation*}

As we have $ d(t) = h(t) $, Theorem \ref{thm:3} states (after a shift of indices), that $ v_{j}^{r} = \frac{j}{r}[t^{r-j}]A(t)^{r} $. So it remains to determine the coefficient of $ t^{r-j} $ in $  \frac{j}{r}A(t)^{r} $. Using the well-known identity:
\begin{equation}
	\label{eq:8}
	\left(\frac{1}{1-x}\right)^{r} = \sum_{\ell=0}^{\infty} \binom{r+\ell-1}{\ell}x^{\ell},
\end{equation}

we thus arrive at the claimed formula by setting $ r -j = \ell $
\begin{equation*}
	v_{j}^{r} = \frac{j}{r}[t^{r-j}]A(t)^{r} = \frac{j}{r}[t^{r-j}]\sum_{\ell=0}^{\infty} \binom{(k-2)r+\ell-1}{\ell}t^{\ell} = \frac{j}{r}\binom{(k-1)r-j-1}{r-j} \tag*{\qedhere}.
\end{equation*}
\end{proof}

\subsection{Characteristic polynomial}

We now proceed to establish a formula for the characteristic polynomial of $ K_{r} $. The following corollary is proved by applying Theorem \ref{thm:13} with $ a_{-1} = 1 $ and $ a_{i-1} = \binom{k+i-3}{k-3} $ for $ i= 2, \ldots, r $.

\begin{corollary}
	\label{thm:34}
	The characteristic polynomial $ k_{r}(\lambda) $ of the matrix $ K_{r} $ satisfies the recurrence relation
	\begin{equation*}
		k_{r}(\lambda) = \left(\binom{k-2}{k-3}-\lambda\right)k_{r-1}(\lambda) + \sum_{i=2}^{r}(-1)^{i+1}\binom{k+i-3}{k-3}k_{r-i}(\lambda).
	\end{equation*}
\end{corollary}

\bigskip

We need the following result to solve the recurrence relation.

\begin{lemma}
	\label{lm:1}
	For any given $ t, m, n \in \mathbb{N} $, it holds that
	\begin{equation*}
		\sum_{j=0}^{n}\binom{j+m-1}{m-1}\binom{m(t+1)}{n-j-t}(-1)^j = \binom{mt}{n-t}.
	\end{equation*}
\end{lemma}

\begin{proof}
	The proof of the lemma is obtained by using two identities related to the binomial coefficient. Using the negation of the upper index of a binomial coefficient, see \textcolor{blue}{\cite[$ \textbf{G} (17) $]{Identities}}, we get that
	\begin{equation*}
		\sum_{j=0}^{n}\binom{j+m-1}{m-1}\binom{m(t+1)}{n-j-t}(-1)^j = \sum_{j=0}^{n-t}(-1)^j\binom{m+j-1}{j}\binom{m(t+1)}{n-t-j} = \sum_{j=0}^{n-t}\binom{-m}{j}\binom{m(t+1)}{n-t-j}.
	\end{equation*}
	
	Finally, applying Chu-Vandermonde identity, see \textcolor{blue}{\cite[$ \textbf{I} (21) $]{Identities}}, we obtain:
	\begin{equation*}
		\sum_{j=0}^{n-t}\binom{-m}{j}\binom{m(t+1)}{n-t-j} = \binom{-m+m(t+1)}{n-t} = \binom{mt}{n-t}.\qedhere
	\end{equation*}
	
\end{proof}
\bigskip

\begin{theorem}
	\label{thm:35}
	The solution of the recurrence relation
	\begin{equation*}
		k_{r}(\lambda) = \left(\binom{k-2}{k-3}-\lambda\right)k_{r-1}(\lambda) - \sum_{i=2}^{r}(-1)^{i}\binom{k+i-3}{k-3}k_{r-i}(\lambda),
	\end{equation*}
	with initial value $ k_{0}(\lambda) = 1 $ is
	\begin{equation*}
		k_{r} = \sum_{\ell=0}^{r}(-1)^\ell\binom{(k-2)(\ell+1)}{r-\ell} \lambda^{\ell}.
	\end{equation*}
\end{theorem}

\begin{proof}
	We use induction on $ r $. The base cases for $ r \leq 2 $ are easily verified. So, assume the theorem holds for all $ i \leq r-1 $. Let then $ i \geq r $ and consider then $ k_{r}(\lambda) $. By induction,
	\small{
	\begin{equation*}
		k_{r}(\lambda) = \left(\binom{k-2}{k-3}-\lambda\right)\sum_{\ell=0}^{r-1}(-1)^{\ell}\binom{(k-2)(\ell+1)}{r-1-\ell}\lambda^{\ell} - \sum_{i=2}^{r}(-1)^{i}\binom{k+i-3}{k-3}\sum_{\ell=0}^{r-i}(-1)^{\ell}\binom{(k-2)(\ell+1)}{r-i-\ell}\lambda^{\ell}.
	\end{equation*}}
	
	We rewrite
	\begin{equation*}
		A = \sum_{i=2}^{r}(-1)^{i}\binom{k+i-3}{k-3}\sum_{\ell=0}^{r-i}(-1)^{\ell}\binom{(k-2)(\ell+1)}{r-i-\ell}\lambda^{\ell}
	\end{equation*}
	
	in the form $ \sum_{\ell=0}^{r}\lambda^{\ell}(-1)^{\ell}c_{\ell} $, where $ c_{\ell} $ is the coefficient of $ \lambda^{\ell}(-1)^{\ell} $ and get
	
	\begin{equation*}
		A = \sum_{\ell=0}^{r}\lambda^{\ell}(-1)^{\ell}\sum_{i=2}^{r}\binom{i+k-3}{k-3}\binom{(k-2)(\ell+1)}{r-i-\ell}(-1)^{i}.
	\end{equation*}

	This further equals, by Lemma \ref{lm:1},
	
	\begin{equation*}
		A = \sum_{\ell=0}^{r}\lambda^{\ell}(-1)^{\ell}\left(\binom{(k-2)\ell}{r-\ell}-\binom{(k-2)(\ell+1)}{r-\ell} + \binom{k-2}{k-3}\binom{(k-2)(\ell+1)}{r-1-\ell}\right).
	\end{equation*}
	
	Then,
	\begin{multline*}
		k_{r}(\lambda) = \left(\binom{k-2}{k-3}-\lambda\right)\sum_{\ell=0}^{r-1}(-1)^{\ell}\binom{(k-2)(\ell+1)}{r-1-\ell}\lambda^{\ell} - A = \sum_{\ell=0}^{r-1}(-1)^{\ell}\binom{k-2}{k-3}\binom{(k-2)(\ell+1)}{r-1-\ell}\lambda^{\ell}- \\
		- \sum_{\ell=0}^{r-1}(-1)^{\ell}\binom{(k-2)(\ell+1)}{r-1-\ell}\lambda^{\ell+1} -\sum_{\ell=0}^{r}\lambda^{\ell}(-1)^{\ell}\bigg(\binom{(k-2)\ell}{r-\ell}-\binom{(k-2)(\ell+1)}{r-\ell} + \binom{k-2}{k-3}\binom{(k-2)(\ell+1)}{r-1-\ell}\bigg).\\
	\end{multline*}

	We make the change of variable $ \ell+1=s $ in the second summation and get:
	\small{
	\begin{multline*}          
		k_{r}(\lambda) = \sum_{\ell=0}^{r}(-1)^{\ell}\binom{k-2}{k-3}\binom{(k-2)(\ell+1)}{r-1-\ell}\lambda^{\ell}-\sum_{s=1}^{r}(-1)^{s-1}\binom{(k-2)s}{r-s}\lambda^{s} - \\
		-\sum_{\ell=0}^{r}\lambda^{\ell}(-1)^{\ell}\left(\binom{(k-2)\ell}{r-\ell}-\binom{(k-2)(\ell+1)}{r-\ell} + \binom{k-2}{k-3}\binom{(k-2)(\ell+1)}{r-\ell-1}\right) =\\
		= \sum_{\ell=0}^{r}(-1)^{\ell}\binom{k-2}{k-3}\binom{(k-2)(\ell+1)}{r-1-\ell}\lambda^{\ell} + \sum_{s=1}^{r}(-1)^{s}\binom{(k-2)s}{r-s}\lambda^{s}		-\sum_{\ell=0}^{r}\lambda^{\ell}(-1)^{\ell}\binom{(k-2)\ell}{r-\ell} + \\
		+\sum_{\ell=0}^{r}\lambda^{\ell}(-1)^{\ell}\binom{(k-2)(\ell+1)}{r-\ell} - \sum_{\ell=0}^{r}(-1)^{\ell}\binom{k-2}{k-3}\binom{(k-2)(\ell+1)}{r-\ell-1}\lambda^{\ell} = \sum_{\ell=0}^{r}\lambda^{\ell}(-1)^{\ell}\binom{(k-2)(\ell+1)}{r-\ell}.\tag*{\qedhere}
	\end{multline*}}
\end{proof}

\bigskip

We now give a formula for the entries of an eigenvector $ x $ associated to an eigenvalue $ \lambda $ of $ K_{r} $. The proof of the result is obtained by applying Theorem \ref{thm:14} with $ a_{-1}=1 $.

\vspace{3.5mm}

\begin{corollary}
	\label{thm:36}
	Let $ x=(x_{r-1},x_{r-2},\ldots,x_{0})^{\top} $ be an eigenvector associated to an eigenvalue $ \lambda $ of $ K_{r} $ for $ r \geq 2 $. Then the entries of the vector $ x $ are of the form:
	
	\begin{equation*}
		x_{i} = (-1)^{i}k_{i}(\lambda)x_{0},
	\end{equation*}
	
	where $ k_{i}(\lambda) $ is the characteristic polynomial of $ K_{i} $.
\end{corollary}

\section{Counting geometric graphs}

In this section, we work with plane geometric graphs whose vertices are represented by the points $ \{p_{1}, \ldots, p_{n}\} $ in convex position in the plane, ordered counter-clockwise. Figures \ref{fig:10} and \ref{fig:7} show several plane geometric graphs.

\subsection{Production matrix}

In a previous paper \textcolor{blue}{\cite{Geometric}}, a production matrix for the number of graphs with given root vertex degree was given. In that case, the degree of the root vertex $ p_{n} $ was defined as the number of edges incident to $ p_{n} $. Also, a formula for the number of geometric graphs was given. However, there are other ways to define the degree of a vertex in the graph.

In this section, the degree of the root vertex is defined in a different way, based on visibility. The \emph{visibility degree} of the root vertex $ p_{n} $ is the number of visible vertices from a vertex $ p_{n+1} $ inserted between $ p_{1} $ and $ p_{n} $ in convex position, minus $ 2 $. Two vertices are visible if the line segment connecting them does not intersect the interior of any edge of the graph. We substract 2 because every vertex in the graph \textquotedblleft sees\textquotedblright \ at least one edge and its endpoints and hence, these vertices can be disregarded.

We start by giving the first level of a generating tree for a given geometric graph in Figure \ref{fig:7}. We generate its children by leaving the new root vertex isolated or connecting it to a subset of vertices visible from the new root vertex.

\vspace{3.5mm}

\begin{center}
	\includegraphics[scale=0.68]{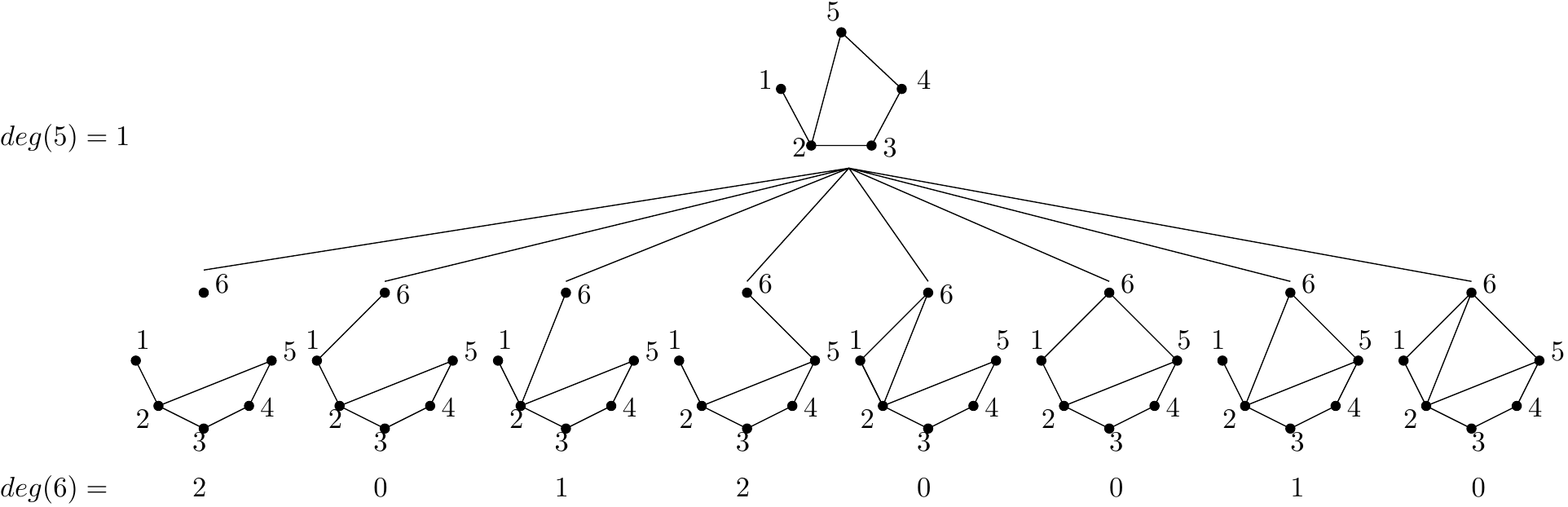}
	\captionof{figure}{Children of a given geometric graph in the tree of geometric graphs.}
	\label{fig:7}
\end{center}

Consider a vertex $ p_{i+1} $ that is inserted in convex position between $ p_{1} $ and $ p_{i} $. Let $ V(p_{i+1}) = \{p_{1}=p_{i+1}^{1}, p_{i+1}^{2}, \ldots, p_{i+1}^{t}=p_{i}\}, \ t < i+1 $, be the sequence, ordered counter-clockwise, of visible vertices from $ p_{i+1} $. We have $ \binom{t}{0} $ ways of leaving $ p_{i+1} $ isolated, $ \binom{t}{1} $ ways of adding an edge from $ p_{i+1} $ to one vertex of $ V(p_{i+1}) $, $ \binom{t}{2} $ ways of adding two edges from $ p_{i+1} $ to a pair of vertices of $ V(p_{i+1}) $, and so on. Thus, the number of children of a given geometric graph with root vertex of visibility degree $ t-2 $ is $ \sum_{k=0}^{t}\binom{t}{k} = 2^{t} $.\bigskip

\begin{theorem}
	\label{thm:38}
	The following $ n \times n $ matrix is a production matrix for geometric graphs of point sets in convex position.
	
	\begin{equation*}
		G_{n} = \begin{pmatrix}
			2 & 4 & 8 & 16 & 32 & \cdots & 2^{n} \\
			2 & 2 & 4 & 8 & 16 & \cdots & 2^{n-1} \\
			0 & 2 & 2 & 4 & 8 & \cdots & 2^{n-2} \\
			0 & 0 & 2 & 2 & 4 & \cdots & 2^{n-3} \\
			0 & 0 & 0 & 2 & 2 & \cdots & 2^{n-4} \\
			\vdots & \vdots & \vdots & \vdots & \vdots & \ddots & \vdots \\
			0 & 0 & 0 & 0 & 0 & \cdots & 2 \\
		\end{pmatrix}.
	\end{equation*}
	
\end{theorem}

\begin{proof}
	Assume that the vector $ v^{i} $ containing the number of geometric graphs for each possible visibility degree of $ p_{i} $ is known. For $ v^{i+1} $, with $ i \leq n $, consider a vertex $ p_{i+1} $ inserted between $ p_{1} $ and $ p_{i} $ in convex position. From the definition of the generating tree of geometric graphs, and the relation $ v^{i+1} = G_{n}v^{i} $ we can determine the entries of $ G_{n} $ as follows:  

\begin{itemize}
	\item \textbf{First row}. For $ p_{i+1} $ to have visibility degree $ 0 $, an edge from $ p_{i+1} $ to $ p_{i+1}^{1} \in V(p_{i+1}) $ must be included. Thus, the number of geometric graphs where $ p_{i+1} $ has visibility degree $ 0 $ is equal to the number of graphs where $ p_{i} $ has visibility degree $ s $, for $ 0 \leq s \leq n-2 $ when adding a subset of edges from $ p_{i+1} $ to some $ p_{i+1}^{j} \in V(p_{i+1}) $, with $ 1 < j \leq t $. This gives a total number of $ \sum_{j=0}^{s+1} \binom{s+1}{j} = 2^{s+1} $ graphs, resulting in a first row of powers of $ 2 $ in $ G_n $.
	
	\item \textbf{Second row}. The number of geometric graphs where $ p_{i+1} $ has visibility degree $ 1 $, obtained from the graphs where $ p_{i} $ has visibility degree $ 0 $, is equal to the two graphs leaving $ p_{i+1} $ isolated or connecting it with $ p_{i} $, thus we get a $ 2 $ in the first column of the second row. As soon as $ p_{i} $ has visibility degree at least $ 1 $, an edge from $ p_{i+1} $ to $ p_{i+1}^{2} \in V(p_{i+1}) $ must be included, and then, we add a subset of edges from $ p_{i+1} $ to some $ p_{i+1}^{j} \in V(p_{i+1}) $ for $ j > 2 $. Thus, the rest of the row is made of powers of $ 2 $.
	
	\item \textbf{Other rows}. The cases where $ |V(p_{i+1})| > |V(p_{i})|+1 $ are not possible, so we get a zero in these cases. The following rows are analogous, shifted by one column every time: in order for $ p_{i+1} $ to have visibility degree $ s $, an edge from $ p_{i+1} $ to $ p_{i+1}^{s+1} \in V(p_{i+1}) $ must be included, and then we can add a subset of edges from $ p_{i+1} $ to some $ p_{i+1}^{j} \in V(p_{i+1}) $, for $ j > s+1 $. So, we get a power of $ 2 $ for each entry. \qedhere
\end{itemize}
\end{proof}\smallskip

Let $ v^{i} $ be a vector of geometric graphs (that is, $ v^{i} $ is the vector obtained by a power of $ G_{n} $ for $ i \leq n $ multiplied by $ v^2 $; the sum of the elements of $ v^{i} $ is the number of geometric graphs with $ i $ vertices). The first vectors are:

\begin{equation*}
	v^{2} = \begin{pmatrix}
		2 \\
		0 \\
		0 \\
		0 \\
		0 \\
		\vdots
	\end{pmatrix} \qquad v^{3} = 
	\begin{pmatrix}
		4 \\
		4 \\
		0 \\
		0 \\
		0 \\
		\vdots
	\end{pmatrix} \qquad v^{4} = 
	\begin{pmatrix}
		24 \\
		16 \\
		8 \\
		0 \\
		0 \\
		\vdots
	\end{pmatrix} \qquad v^{5} = 
	\begin{pmatrix}
		176 \\
		112 \\
		48 \\
		16 \\
		0 \\
		\vdots
	\end{pmatrix}
\end{equation*}

\bigskip

If we sum up all the entries of each vector $ v^{i} $, we can verify that the sequence of numbers coincides with the one given in \textcolor{blue}{\cite{Noy, Geometric}} and corresponds to Sequence $ A054726 $ in the OEIS \textcolor{blue}{\cite{OEIS}} for the number of graphs with $ n $ nodes on a circle without crossing edges.

\vspace{3.5mm}

Next, we give a formula to calculate the entries of the vector $ v^{n} $ of geometric graphs different from the one obtained in \textcolor{blue}{\cite{Geometric}}. The difference comes from our definition of the degree of the root vertex. The proof of the theorem is similar to the one for Theorem \ref{thm:32}, so it is deferred to Appendix \ref{A.1}.

\begin{theorem}
	\label{thm:37}
	The number of geometric graphs with $ n $ vertices where the root vertex has visibility degree $ j-1 $, $ \forall j=1,\ldots,n-1 $ is:
	\begin{equation*}
		v_{j}^{n} = \frac{j}{n-1}2^{n-1-j}\sum_{k=j}^{n-1}\binom{n-1}{k}\binom{n+k-j-2}{k-j}(-1)^{n-1-k}2^{k}.
	\end{equation*}
\end{theorem}

\subsection{Characteristic polynomial}

Let $ g_n(\lambda) $ be the characteristic polynomial of $ G_n $. The sequence $ \{g_{n}(\lambda)\} $ starts with

\begin{equation*}
	\begin{tabular}{c|c}
		$ n = 1 $ & $ -\lambda+2 $ \\
		$ n = 2 $ & $ \lambda^{2}-4\lambda-4 $ \\ 
		$ n = 3 $ & $ -\lambda^{3}+6\lambda^{2}+4\lambda+8 $ \\ 
		$ n = 4 $ & $ \lambda^{4}-8\lambda^{3}-16 $ \\
		$ n = 5 $ & $ -\lambda^{5}+10\lambda^{4}-8\lambda^{3}-16\lambda^{2}-16\lambda+32 $ \\
		$ n = 6 $ & $ \lambda^{6}-12\lambda^{5}+20\lambda^{4}+32\lambda^{3}+48\lambda^{2}+64\lambda-64 $
	\end{tabular}
\end{equation*}
\smallskip

\begin{corollary}
	\label{thm:40}
	The characteristic polynomial $ g_{n}(\lambda) $ of the matrix $ G_{n} $ satisfies the recurrence relation
	\begin{equation*}
		g_{n}(\lambda) = (2-\lambda)g_{n-1}(\lambda) + \sum_{i=2}^{n}(-1)^{i+1}2^{2i-1}g_{n-i}(\lambda).
	\end{equation*}
\end{corollary}

\smallskip
Corollary \ref{thm:40} follows from Theorem \ref{thm:13} with $ a_{-1} = 2, \ a_0 = 2 $ and $ a_{i-1} = 2^{i}$.

\begin{theorem}
	\label{thm:41}
	The solution of the recurrence relation
	\begin{equation*}
		g_{n}(\lambda) = (2-\lambda)g_{n-1}(\lambda) + \sum_{i=2}^{n}(-1)^{i+1}2^{2i-1}g_{n-i}(\lambda)
	\end{equation*}
	with initial condition $ g_{0}(\lambda) = 1 $ is
	\begin{equation*}
		g_{n}(\lambda) = \sum_{t=0}^{n}\sum_{k=t}^{n}(-1)^k\binom{k}{t}\binom{t+1}{n-k}2^{2n-t-k}\lambda^{t}.
	\end{equation*}
\end{theorem}

\begin{proof}
	Consider the infinite matrix

\begin{equation*}
	M = \begin{pmatrix}
		2-\lambda & -2^3 & 2^5 & -2^7 & \cdots \\
		1 & 0 & 0 & 0 & \cdots \\
		0 & 1 & 0 & 0 & \cdots \\
		0 & 0 & 1 & 0 & \cdots \\
		0 & 0 & 0 & 1 & \cdots \\
		\vdots & \vdots & \vdots & \vdots & \ddots \\
	\end{pmatrix}
\end{equation*}

Let $ w_{0} $ be the vector $ (1,0,\ldots)^{\top} $, and let $ w_{i} $ be the vector whose first $ i $ entries are the first $ i $ characteristic polynomials $ g_{i}(\lambda) $, and the remaining ones are zero, i.e., $ w_{i} = (g_{i}(\lambda), g_{i-1}(\lambda), g_{i-2}(\lambda), \ldots, g_{1}(\lambda), 0, \ldots)^{\top} $. Then $ M \cdot w_{i} = w_{i+1} $. It follows that $ w_{n} $ is the first column of $ M^{n} $. We can now use the Riordan Array approach. The Z-sequence is $ \{2-\lambda, -8, 32, \ldots\} $ with  generating function $ Z(t) = -\lambda + \frac{2}{1+4t}$. The A-sequence is $ \{1,0,\ldots\} $ with generating function $ A(t) = 1 $. It follows that $ h(t) = 1 $ and

\begin{equation*}
	d(t) = \frac{1}{1-t(-\lambda + \frac{2}{1+4t})} = \frac{1+4t}{1+t(2+\lambda(1+4t))} = d_{1}(t)+4td_{1}(t),
\end{equation*}

where
\begin{equation*}
	d_{1}(t) = \frac{1}{1+t(2+\lambda(1+4t))}.
\end{equation*}

Then, by Equation (\ref{eq:4}),
\begin{multline}
\label{eq:35}
	d_{n,j} = [t^n]d(t)(th(t))^{j} = [t^n]d(t)t^{j} = [t^{n-j}]d(t) =  [t^{n-j}](d_{1}(t)+4td_{1}(t)) =\\
	= [t^{n-j}]d_{1}(t) + 4[t^{n-1-j}]d_{1}(t).
\end{multline}

Using the well-known identity
\begin{equation}
	\label{eq:5}
	\frac{1}{1+z} = \sum_{k=0}^{\infty}(-1)^{k}z^{k},
\end{equation}

we know that
\begin{equation*}
	d_{1}(t)=\sum_{k=0}^{\infty}(-1)^kt^{k}(2 + \lambda(1+4t))^{k}.
\end{equation*}

We apply the binomial theorem to $ (2 + \lambda(1+4t))^{k} $. We then only have to take the coefficient of $ t^{n-j} $ in $ d_{1}(t) $ for the left term of Equation (\ref{eq:35}) and the coefficient of $ t^{n-1-j} $ in $ d_{1}(t) $ for the right term of Equation (\ref{eq:35}). After some short calculations, this gives the formula
\begin{equation*}
	g_{n}(\lambda) = \sum_{k=0}^{n}\sum_{j=0}^{n}(-1)^{k}\binom{k}{j}\binom{k-j}{n-k}2^{2(n-k)+j}\lambda^{k-j} + \sum_{k=0}^{n}\sum_{j=0}^{n}(-1)^{k}\binom{k}{j}\binom{k-j}{n-k-1}2^{2(n-k)+j}\lambda^{k-j}.
\end{equation*}

We finally rewrite $ g_{n}(\lambda) = \sum_{t=0}^{n} c_{t}\lambda^{t} $, with $ c_{t} $ the coefficient of $ \lambda^{t} $. Since $ \lambda^{t} = \lambda^{k-j} $ we set $ j = k - t $ and get
\begin{equation*}
	g_{n}(\lambda) = \sum_{t=0}^{n}\left(\sum_{k=0}^{n}(-1)^{k}\binom{k}{k-t}\binom{t}{n-k}2^{2n-t-k} + \sum_{k=0}^{n}(-1)^{k}\binom{k}{k-t}\binom{t}{n-k-1}2^{2n-t-k}\right)\lambda^{t}.
\end{equation*}

Note that $ k \geq t $, so after some elementary operations, taking into account that $ \binom{t}{n-k}+\binom{t}{n-k-1} = \binom{t+1}{n-k} $, and the symmetry identity for the binomial coefficient $ \binom{k}{k-t} $, we finally get the formula

\begin{equation*}
	g_{n}(\lambda) = \sum_{t=0}^{n}\sum_{k=t}^{n}(-1)^k\binom{k}{t}\binom{t+1}{n-k}2^{2n-t-k}\lambda^{t}. \tag*{\qedhere}
\end{equation*}
\end{proof}\bigskip

The proof of the next corollary is obtained by applying Theorem \ref{thm:14} with $ a_{-1} = 2 $.

\begin{corollary}
	\label{thm:42}
	Let $ x=(x_{n-1},x_{n-2},\ldots,x_{0})^{\top} $ be an eigenvector associated to an eigenvalue $ \lambda $ of $ G_{n} $ for $ n \geq 2 $. Then the entries of the vector $ x $ are of the form:
	
	\begin{equation*}
		x_{i} = \left(-\frac{1}{2}\right)^{i}g_{i}(\lambda)x_{0},
	\end{equation*}
	
	where $ g_{i}(\lambda) $ is the characteristic polynomial of $ G_{i} $.
\end{corollary}

\section{Counting connected graphs}

In this section we treat plane connected graphs with vertices in convex position. Recall that a graph is \textit{connected} if there is a path between every pair of vertices. Figure \ref{fig:15} shows several connected graphs.

\subsection{Production matrix}

In this section, we present a production matrix for connected graphs where the degree of the root vertex is again defined based on visibility. As far as we know, production matrices for connected graphs have not been considered before. In general, for a given connected graph, we generate its children by using the following mapping. Consider a vertex $ p_{i+1} $ that is inserted in convex position between $ p_{1} $ and $ p_{i} $. Let, as before, $ V(p_{i+1}) = \{p_{1}=p_{i+1}^1, p_{i+1}^{2}, \ldots, p_{i+1}^{t}=p_{i}\}, \ t < i+1 $, be the sequence of visible vertices from the new vertex $ p_{i+1} $, ordered counter-clockwise. We obtain a connected graph by adding edges from $ p_{i+1} $ to all the visible vertices in the ordered sequence $ V' = \{p_{i+1}^{k}, \ldots, p_{i+1}^{j}\} $ with $ p_{i+1}^{\ell} \in V(p_{i+1}) $ for each $ k \leq \ell \leq j $, and removing a subset of the edges, such that both endpoints are in $ V' $. We need to connect $ p_{i+1} $ with consecutive vertices in $ V' $ so that the connected graphs are generated only once. An example of how we can apply the mapping to generate the children of a given connected graph can be seen in Figure \ref{fig:15}.

\begin{center}
	\includegraphics[scale=0.7]{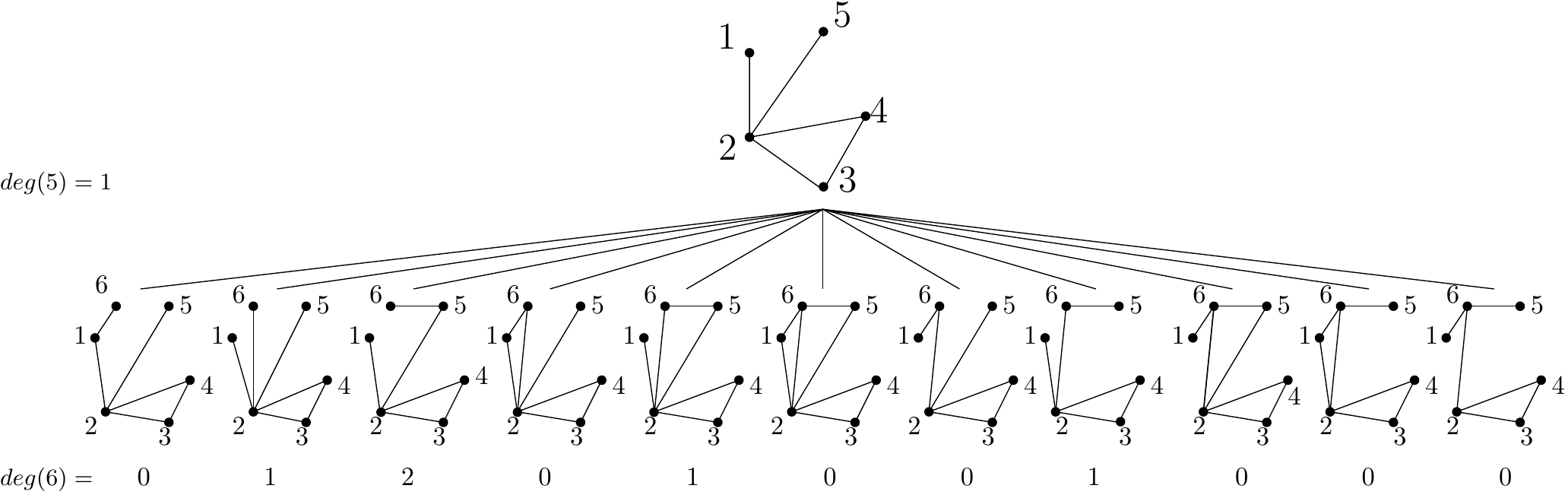}
	\captionof{figure}{Children of a given connected graph in the tree of connected graphs.}
	\label{fig:15}
\end{center}

\bigskip

\begin{theorem}
	\label{thm:6}
	The following $ n \times n $ matrix is a production matrix for plane connected graphs of point sets in convex position.
	
	\begin{equation*}
		C_{n} = \begin{pmatrix}
			3 & 7 & 15 & 31 & 63 & \cdots & 2^{n+1}-1 \\
			1 & 3 & 7 & 15 & 31 & \cdots & 2^{n}-1 \\
			0 & 1 & 3 & 7 & 15 & \cdots & 2^{n-1}-1 \\
			0 & 0 & 1 & 3 & 7 & \cdots & 2^{n-2}-1 \\
			0 & 0 & 0 & 1 & 3 & \cdots & 2^{n-3}-1 \\
			\vdots & \vdots & \vdots & \vdots & \vdots & \ddots & \vdots \\
			0 & 0 & 0 & 0 & 0 & \cdots & 3 \\
		\end{pmatrix}.
	\end{equation*}
	
\end{theorem}

\begin{proof}
	Assume that the vector $ v^{i} $, containing the number of connected graphs for each possible visibility degree of $ p_{i} $, is known. For $ v^{i+1} $, with $ i \leq n $, consider a vertex $ p_{i+1} $ inserted between $ p_{1} $ and $ p_{i} $ in convex position. From the definition of the generating tree of connected graphs, and the relation $ v^{i+1}=C_{n}v^{i} $ we can determine the entries of $ C_{n} $ as follows:

\begin{itemize}
	\item \textbf{First row}. For $ p_{i+1} $ to have visibility degree $ 0 $, an edge from $ p_{i+1} $ to $ p_{i+1}^{1} $ must be included. Thus, the number of connected graphs where $ p_{i+1} $ has visibility degree $ 0 $ is equal to the number of connected graphs where $ p_{i} $ has visibility degree $ s $, for $ 0 \leq s \leq n-2 $ when adding some edges from $ p_{i+1} $ to all the visible vertices from $ p_{1} $ to $ p_{i+1}^{j} \in V(p_{i+1}) $, where $ 1 \leq j \leq s+2 $, removing a subset of edges with both endpoints between $ p_{1} $ and $ p_{i+1}^{j} \in V(p_{i+1}) $. This gives a total number of $ \sum_{r=0}^{s+1}2^{r} = 2^{s+2} - 1 $ graphs, resulting in a first row of powers of $ 2 $ minus $ 1 $ in $ C_n $.
	
	\item \textbf{Second row}. The number of connected graphs where $ p_{i+1} $ has visibility degree $ 1 $ and $ p_{i} $ has visibility degree $ 0 $ is equal to the connected graph obtained by connecting $ p_{i+1} $ to $ p_{i} $, thus we get a one in the first column of the second row. As soon as $ p_{i} $ has visibility degree at least $ s \geq 1 $, we have to add one edge from $ p_{i+1} $ to $ p_{i+1}^{2} $, and then we add edges from $ p_{i+1} $ to all the visible vertices from $ p_{i+1}^{2} $ to $ p_{i+1}^{j} $ in $ V(p_{i+1}) $, where $ j \geq 2 $, removing a subset of edges with both endpoints between $ p_{i+1}^{2} $ and $ p_{i+1}^{j} \in V(p_{i+1}) $. This leads to $ \sum_{r=0}^{s}2^{r} = 2^{s+1} - 1 $ graphs. Thus, the rest of the row is made of powers of $ 2 $ minus $ 1 $.
	
	\item \textbf{Other rows}. The cases when $ |V(p_{i+1})| > |V(p_{i})|+1 $ are not possible, so we get a zero in these cases. The following rows are analogous, shifted by one column every time: in order for $ p_{i+1} $ to have visibility degree $ s $, one edge needs to be added from $ p_{i+1} $ to $ p_{i+1}^{s+1} \in V(p_{i+1}) $, and then we can add a subset of edges from $ p_{i+1} $ to all the visible vertices from $ p_{i+1}^{s+1} $ to $ p_{i+1}^{j} $ in $ V(p_{i+1}) $, where $ j \geq s+1 $, removing a subset of edges with both endpoints between $ p_{i+1}^{s+1} $ and $ p_{i+1}^{j} $ in $ V(p_{i+1}) $. So, we get a power of $ 2 $ minus $ 1 $ for each entry. \qedhere
\end{itemize}
\end{proof}

\bigskip

Let $ v^{i} $ be the vector of connected graphs (that is, $ v^{i} $ is the first column of a power of $ C_{n} $ for $ i \leq n $; the sum of the elements of $ v^{i} $ is the number of connected graphs). The first vectors are:

\begin{equation*}
	v^{2} = \begin{pmatrix}
		1 \\
		0 \\
		0 \\
		0 \\
		0 \\
		\vdots
	\end{pmatrix} \qquad v^{3} = 
	\begin{pmatrix}
		3 \\
		1 \\
		0 \\
		0 \\
		0 \\
		\vdots
	\end{pmatrix} \qquad v^{4} = 
	\begin{pmatrix}
		16 \\
		6 \\
		1 \\
		0 \\
		0 \\
		\vdots
	\end{pmatrix} \qquad v^{5} = 
	\begin{pmatrix}
		105 \\
		41 \\
		9 \\
		1 \\
		0 \\
		\vdots
	\end{pmatrix}
\end{equation*}

If we sum up all the entries of each vector $ v^{i} $, we can verify that the sequence of numbers is given as Sequence $ A007297 $ in the OEIS \textcolor{blue}{\cite{OEIS}}, obtaining the number of connected graphs with $ i $ labelled nodes on a circle with straight-line edges that do not cross.

\vspace{3.5mm}

Next, we give a formula for the entries of a vector $ v^{n} $ of connected graphs. Each number $ v_{j}^{n} $ counts the number of connected graphs with $ n $ points in convex position with root vertex of visibility degree $ j-1 $. The proof of the theorem is similar to the one of Theorem \ref{thm:32}, so it is deferred to Appendix \ref{A.2}.

\bigskip

\begin{theorem}
	\label{thm:7}
	The number of connected graphs with $ n $ vertices, where the root vertex has visibility degree $ j-1 $, $ \forall j=1,\ldots, n-1 $ is:
	\begin{equation*}
		v_{j}^{n} = \frac{j}{n-1}2^{n-1}\sum_{k=0}^{n-1}\binom{n-1}{k}\left(-\frac{1}{2}\right)^{k}\sum_{\ell= 0}^{n-j-1}\binom{n-2-k+\ell}{\ell}\binom{k+n-\ell-j-2}{n-\ell-j-1}2^{\ell}.
	\end{equation*}
\end{theorem}

\subsection{Characteristic polynomial}

Let $ c_n(\lambda) $ be the characteristic polynomial of $ C_n $. The sequence $ \{c_{n}(\lambda)\} $ starts with

\begin{equation*}
	\begin{tabular}{c|c}
		$ n = 1 $ & $ -\lambda+3 $ \\
		$ n = 2 $ & $ \lambda^{2}-6\lambda+2 $ \\ 
		$ n = 3 $ & $ -\lambda^{3}+9\lambda^{2}-13\lambda $ \\ 
		$ n = 4 $ & $ \lambda^{4}-12\lambda^{3}+33\lambda^{2}-12\lambda $ \\
		$ n = 5 $ & $ -\lambda^{5}+15\lambda^{4}-62\lambda^{3}+63\lambda^{2}-4\lambda $ \\
		$ n = 6 $ & $ \lambda^{6}-18\lambda^{5}+100\lambda^{4}-180\lambda^{3}+66\lambda^{2} $
	\end{tabular}
\end{equation*}

\begin{corollary}
	\label{thm:43}
	The characteristic polynomial $ c_{n}(\lambda) $ of the matrix $ C_{n} $ satisfies the recurrence relation
	\begin{equation*}
		c_{n}(\lambda) = (3-\lambda)c_{n-1}(\lambda) - \sum_{i=2}^{n}(-1)^{i}(2^{i+1}-1)c_{n-i}(\lambda).
	\end{equation*}
\end{corollary}

Corollary \ref{thm:43} can be derived from Theorem \ref{thm:13} with $ a_{-1}=1, \ a_0=3 $, and $ a_{i-1} = 2^{i+1}-1 $.

\begin{theorem}
	\label{thm:8}
	The solution of the recurrence relation
	\begin{equation*}
		c_{n}(\lambda) = (3-\lambda)c_{n-1}(\lambda) - \sum_{i=2}^{n}(-1)^{i}(2^{i+1}-1)c_{n-i}(\lambda)
	\end{equation*}
	with initial condition $ c_{0}(\lambda) = 1 $ is
	\begin{equation*}
		c_{n}(\lambda) = \sum_{t=0}^{n}\left(\sum_{\ell=0}^{t}\binom{t}{\ell}(-1)^{t}2^{t-\ell}3^{n+2\ell-3t-2}\left[2\binom{\ell}{n+2\ell-3t-2}+9\binom{\ell+1}{n+2\ell-3t}\right]\right) \lambda^{t}.
	\end{equation*}
\end{theorem}

\bigskip

The proof of Theorem \ref{thm:8} is similar to the one of Theorem \ref{thm:41}, so it is deferred to Appendix \ref{B.1}. To conclude, we give a formula for the entries of an eigenvector $ x $ associated to the eigenvalue $ \lambda $ of $ C_{n} $. We derive Corollary \ref{thm:9} from Theorem \ref{thm:14} with $ a_{-1} = 1 $.

\begin{corollary}
	\label{thm:9}
	Let $ x=(x_{n-1},x_{n-2},\ldots,x_{0})^{\top} $ be an eigenvector associated to an eigenvalue $ \lambda $ of $ C_{n} $ for $ n \geq 2 $. Then the entries of the vector $ x $ are of the form:
	
	\begin{equation*}
		x_{i} = (-1)^{i}c_{i}(\lambda)x_{0},
	\end{equation*}
	
	where $ c_{i}(\lambda) $ is the characteristic polynomial of $ C_{i} $.
\end{corollary}\bigskip

\subsection{Relation between geometric graphs and connected geometric graphs}

In this section we establish a relation between the class of connected geometric graphs and the class of all geometric graphs. To that end, we start by creating a production matrix for geometric graphs whose entries are a function of the numbers $ c_{i} $ of connected graphs with $ i $ vertices.

It will be convenient to use a different definition of the root vertex degree. Let $ \{p_{1}, p_{2},\ldots, p_{n}\} $ be the set of vertices of a geometric graph, ordered counter-clockwise, in convex position. The \textit{isolation degree} of the root vertex $ p_{n} $ is defined as the number of isolated visible vertices from a vertex $ p_{n+1} $ inserted between $ p_{1} $ and $ p_{n} $ in convex position.

In general, for a given geometric graph, we generate its children by using the following mapping. Consider a vertex $ p_{i+1} $ that is inserted in convex position between $ p_1 $ and $ p_i $. Let $ I(p_{i+1})= \{p_1=p_{i+1}^{1}, p_{i+1}^{2}, \ldots, p_{i+1}^{t}=p_i\}, \ t < i+1 $, be the sequence of isolated visible vertices from $ p_{i+1} $, ordered counter-clockwise. We obtain a geometric graph by leaving the new root vertex isolated or by creating one connected component with the new root vertex $ p_{i+1} $ and a subset of vertices from $ I(p_{i+1}) $. The rest of the vertices from $ I(p_{i+1}) $ that do not belong to such a component remain isolated. An example of how we generate the children of a given geometric graph is shown in Figure \ref{fig:16}.

\begin{center}
	\includegraphics[scale=0.65]{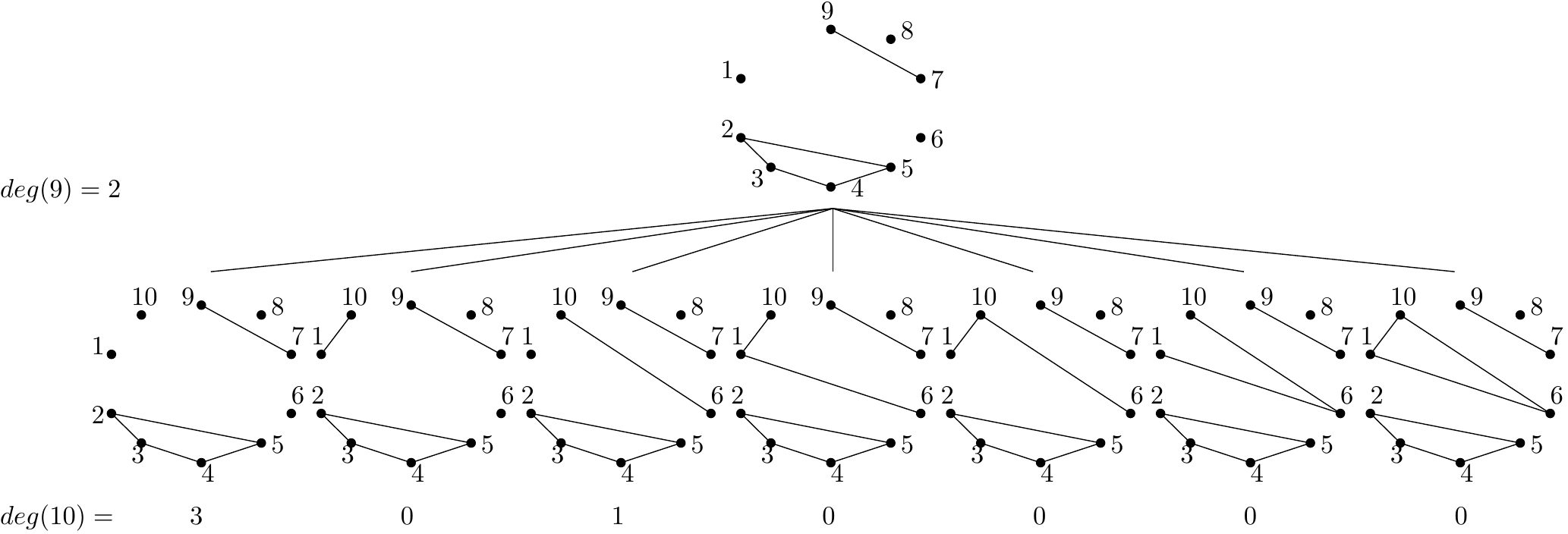}
	\captionof{figure}{Children of a given geometric graph in the tree of geometric trees.}
	\label{fig:16}
\end{center}

\bigskip

\begin{theorem}
	\label{thm:10}
	The following $ n \times n $ matrix is a production matrix for plane geometric graphs of point sets in convex position.
	
	\begin{equation*}
		R_{n} = \begin{pmatrix}
			0 & a_{2} & a_{3} & a_{4} & a_{5} & \cdots & a_{n} \\
			1 & 0 & a_{2} & a_{3} & a_{4} & \cdots & a_{n-1} \\
			0 & 1 & 0 & a_{2} & a_{3} & \cdots & a_{n-2} \\
			0 & 0 & 1 & 0 & a_{2} & \cdots & a_{n-3} \\
			0 & 0 & 0 & 1 & 0 & \cdots & a_{n-4} \\
			\vdots & \vdots & \vdots & \vdots & \vdots & \ddots & \vdots \\
			0 & 0 & 0 & 0 & 0 & \cdots & 0
		\end{pmatrix}.
	\end{equation*}
	
	where $ a_{j} = \sum_{i=2}^{j} \binom{j-2}{i-2}c_{i} $ for $ j \geq 2 $ and $ c_{i} $ is the number of connected graphs with $ i $ vertices for $ i=2,\ldots, n $.
	
\end{theorem}

\begin{proof}
	Assume that the vector $ v^{i} $, containing the number of connected graphs for each possible isolation degree of $ p_{i} $, is known. For $ v^{i+1} $, with $ i \leq n $, consider a vertex $ p_{i+1} $ inserted between $ p_{1} $ and $ p_{i} $ in convex position. From the definition of the generating tree and the relation $ v^{i+1}=R_{n}v^{i} $ we can determine the entries of $ R_{n} $ as follows:

\begin{itemize}
	\item \textbf{Main diagonal}. Observe first that, when adding the vertex $ p_{i+1} $ to a geometric graph with isolation degree of the root vertex $ p_{i} $ equal to $ s $, we cannot obtain a geometric graph with root vertex of isolation degree $ s $. This gives a main diagonal of $ 0 $s in $ R_{n} $.
	
	\item \textbf{First subdiagonal}. We may obtain one geometric graph with $ |I(p_{i+1})| = s+1 $ from a geometric graph with $ |I(p_{i})| = s $ by leaving the new vertex isolated. This gives the first subdiagonal of $ 1 $s.

	\item \textbf{Other diagonals}. The cases where $ |I(p_{i+1})| > |I(p_{i})|+1 $ are not possible, so we get a $ 0 $ in these cases. In general, for any geometric graph with $ i $ vertices, we can decide whether to keep $ p_{i+1} $ isolated or to connect it to a subset of vertices of $ I(p_{i+1}) $ as follows: if we want  $ p_{i+1} $ to have isolation degree $ s $, with $ s < t $, we have to connect $ p_{i+1} $ to $ p_{i+1}^{s+1} \in I(p_{i+1}) $ in $ c_2 $ different ways, and then we have $ \binom{t-s-1}{1} c_{3} $ ways of connecting $ p_{i+1} $ and $ p_{i+1}^{s+1} $ to $ p_{i+1}^{j} $, for each $ j> s+1 $ in $ I(p_{i+1}) $, we also have $ \binom{t-s-1}{2}c_{4} $ ways of connecting $ p_{i+1} $ and $ p_{i+1}^{s+1} $ to a pair of vertices $ p_{i+1}^{j}, p_{i+1}^{k} $, for each pair of $ j, \ k $ such that $ s+1 < j < k $ in $ I(p_{i+1}) $, and so on. This yields the claimed matrix. \qedhere
\end{itemize}
\end{proof}
\bigskip

Substituting each entry $ c_i $ by the number of connected graphs on $ i $ vertices gives us a new production matrix that enumerates the number of geometric graphs, see Table \ref{fig:14}(a). In addition, if we substitute in $ R_n $ each entry $ c_i $ by the number of spanning trees on $ i $ vertices, we obtain the matrix in Table \ref{fig:14}(b), which is a production matrix for non-crossing forests, or if we substitute in $ R_n $ the entries $ c_i $ by the number of spanning paths with $ i $ vertices, we obtain the matrix in Table \ref{fig:14}(c), which is a production matrix for non-crossing forest of paths.

\section{Counting non-crossing partitions}

Given a set of $ n $ elements $ [n] = \{1, 2, \ldots, n\} $, a partition of $ [n] $ is a family of nonempty, pairwise disjoint sets $ B_{1}, B_{2}, \ldots, B_{k} $, called blocks, whose union is the $ n $-element set, see graphs from Figure \ref{fig:20} and recall that each cycle in a graph represents a block. In our case, we treat partitions with vertices in convex position. A partition of $ [n] $ is non-crossing if whenever four elements, $ 1 \leq a < b < c < d \leq n $, are such that if $ a, c $ are in the same block and $ b, d $ are in the same block, then the two blocks coincide.

\subsection{Production matrix}

There exist previous papers where a production matrix of non-crossing partitions has been given \textcolor{blue}{\cite{Geometric}}. The matrix given in \textcolor{blue}{\cite{Geometric}}, in fact, is the same as for triangulations, see Table \ref{fig:12}(a).

In this section we give a new production matrix for this type of graphs. Let $ \{p_{1},p_{2},\ldots,p_{n}\} $ be the set of vertices of a non-crossing  partition, ordered counter-clockwise, in convex position. While in \textcolor{blue}{\cite{Geometric}}, the degree of the root vertex $ p_{n} $ is defined as the number of blocks visible from the root vertex, here, we use the definition of isolation degree.

Let $ I(p_{i+1}) = \{p_{i+1}^{1}, p_{i+1}^{2}, \ldots, p_{i+1}^{t}\}, \ t < i+1 $, be the sequence of isolated visible vertices from $ p_{i+1} $, ordered counter-clockwise. For a given non-crossing partition, we generate its children by leaving the new root vertex $ p_{i+1} $ isolated or by connecting it with a subset of its visible vertices obtaining a new block between them. Figure \ref{fig:20} shows an example of how the children of a non-crossing partition are generated.

\begin{center}
	\includegraphics[scale=0.59]{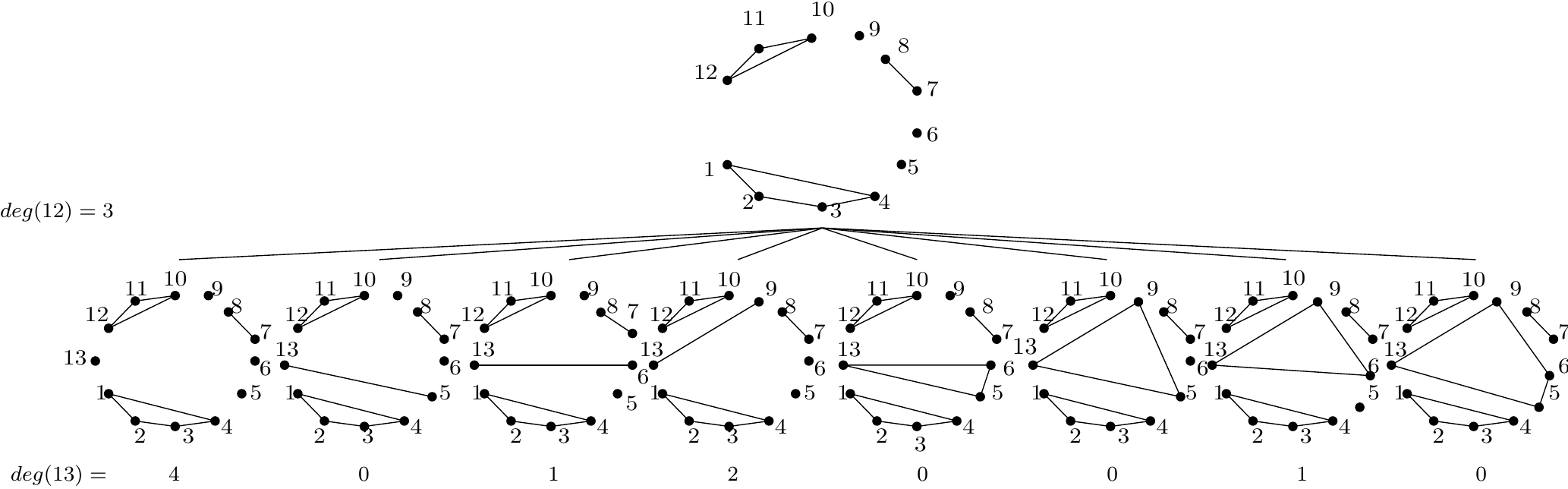}
	\captionof{figure}{Children of a given non-crossing partition in the tree of non-crossing partitions.}
	\label{fig:20}
\end{center}

\begin{theorem}
	\label{thm:1}
	The following $ n \times n $ matrix is a production matrix for non-crossing partitions of point sets in convex position.
	
	\begin{equation*}
		B_{n} = \begin{pmatrix}
			0 & 1 & 2 & 4 & 8 & \cdots & 2^{n-2} \\
			1 & 0 & 1 & 2 & 4 & \cdots & 2^{n-3} \\
			0 & 1 & 0 & 1 & 2 & \cdots & 2^{n-4} \\
			0 & 0 & 1 & 0 & 1 & \cdots & 2^{n-5} \\
			0 & 0 & 0 & 1 & 0 & \cdots & 2^{n-6} \\
			\vdots & \vdots & \vdots & \vdots & \vdots & \ddots & \vdots \\
			0 & 0 & 0 & 0 & 0 & \cdots & 0
		\end{pmatrix}.
	\end{equation*}
	
\end{theorem}

\begin{proof}
	Assume that the vector $ v^{i} $, containing the number of non-crossing partitions for each possible isolation degree of $ p_{i} $, is known. For $ v^{i+1} $, with $ i \leq n $, consider a vertex $ p_{i+1} $ inserted between $ p_{1} $ and $ p_{i} $ in convex position. From the definition of the generating tree and the relation $ v^{i+1}=B_{n}v^{i} $ we can determine the entries of $ B_{n} $ as follows:

\begin{itemize}
	\item \textbf{Main diagonal}. Observe first that, when adding the point $ p_{i+1} $ to a non-crossing partition with isolation degree of the root vertex $ p_{i} $ equal to $ s $, we cannot obtain a non-crossing partition with root vertex of isolation degree $ s $. This gives a main diagonal of $ 0 $s in $ B_{n} $.
	
	\item \textbf{First subdiagonal}. If $ |I(p_{i+1})| = |I(p_{i})|+1 $, we obtain one non-crossing partition by leaving the new root vertex $ p_{i+1} $ isolated. This gives a first subdiagonal of $ 1 $s in $ B_n $.
	
	\item \textbf{First row}. The number of non-crossing partitions where $ p_{i+1} $ has isolation degree $ 0 $ is equal to the number of non-crossing partitions where $ p_{i} $ has isolation degree $ s $, for $ 0 \leq s \leq n $. Those partitions are obtained by adding an edge from $ p_{i+1} $ to $ p_{1} $; and then, we can add a subset of isolated visible vertices from $ p_{i+1} $ to the block containing $ p_{i+1} $. We have $ \binom{s-1}{0} $ ways of producing a block of size $ 2 $ with the vertices $ p_1 $ and $ p_{i+1} $, $ \binom{s-1}{1} $ blocks of size $ 3 $ with $ p_1, \ p_{i+1} $ and one vertex of $ I(p_{i+1}) $, and so on. This results in $ \sum_{j=0}^{s-1}\binom{s-1}{j} = 2^{s-1} $ graphs. This gives a first row of powers of $ 2 $ in $ B_n $.
	
	\item \textbf{Other rows}. The cases where $ |I(p_{i+1})| > |I(p_{i})|+1 $ are not possible, so we get a zero in these cases. The following rows are analogous, shifted by one column every time: in order for $ p_{i+1} $ to have isolation degree $ s $, one edge needs to be added from $ p_{i+1} $ to $ p_{i+1}^{s+1} $, and then we connect $ p_{i+1} $ to a subset of its isolated visible vertices obtaining a block. So we get a power of $ 2 $ for each entry.\qedhere
\end{itemize}
\end{proof}

\bigskip

Let $ v^{i} $ be a vector of non-crossing partitions (that is, $ v^{i} $ is the first column of a power of $ B_{n} $ for $ i \leq n $; the sum of the elements of $ v^{i} $ is the number of non-crossing partitions). The first vectors are:

\begin{equation*}
	v^{1} = 
	\begin{pmatrix}
		0 \\
		1 \\
		0 \\
		0 \\
		0 \\
		\vdots
	\end{pmatrix} \qquad v^{2} = 
	\begin{pmatrix}
		1 \\
		0 \\
		1 \\
		0 \\
		0 \\
		\vdots
	\end{pmatrix} \qquad v^{3} = 
	\begin{pmatrix}
		2 \\
		2 \\
		0 \\
		1 \\
		0 \\
		\vdots
	\end{pmatrix} \qquad v^{4} = 
	\begin{pmatrix}
		6 \\
		4 \\
		3 \\
		0 \\
		1 \\
		\vdots
	\end{pmatrix}
\end{equation*}

If we sum up all the entries of each vector $ v^{i} $, we can verify that the numbers obtained are the Catalan numbers, the same sequence implied by production matrix (a) in Table \ref{fig:11} that counts the number of triangulations with vertices in convex position, as well as the number of non-crossing perfect matchings. Thanks to the new production matrix, we are able to partition Catalan numbers in a different way. So we present a new formula for the vectors $ v^{i} $, whose entries are different from the well-known Ballot numbers. The proof of the theorem is similar to the one of Theorem \ref{thm:32}, so it is deferred to Appendix \ref{A.3}.

\begin{theorem}
	\label{thm:2}
	The number of non-crossing partitions with $ n $ vertices, where the root vertex has isolation degree $ j-1, \ \forall j=1,\ldots,n+1 $ is:
	\begin{equation*}
		v_{j}^{n} = \frac{j}{n+1}\sum_{k=\big\lceil\frac{n+j+1}{2}\big\rceil}^{n+1}\binom{n+1}{k}\binom{k-j-1}{2k-n-j-1}2^{2k-n-j-1}.
	\end{equation*}
\end{theorem}

\subsection{Characteristic polynomial}

Let $ b_n(\lambda) $ be the characteristic polynomial of $ B_n $. The sequence $ \{b_{n}(\lambda)\} $ starts with

\begin{equation*}
	\begin{tabular}{c|c}
		$ n = 1 $ & $ -\lambda $ \\
		$ n = 2 $ & $ \lambda^{2}-1 $ \\ 
		$ n = 3 $ & $ -\lambda^{3}+2\lambda+2 $ \\ 
		$ n = 4 $ & $ \lambda^{4}-3\lambda^{2}-4\lambda-3 $ \\
		$ n = 5 $ & $ -\lambda^{5}+4\lambda^{3}+6\lambda^{2}+5\lambda+4 $ \\
		$ n = 6 $ & $ \lambda^{6}-5\lambda^{4}-8\lambda^{3}-6\lambda^{2}-4\lambda-5 $
	\end{tabular}
\end{equation*}\bigskip

\begin{corollary}
	\label{thm:50}
	The characteristic polynomial $ b_{n}(\lambda) $ of the matrix $ B_{n} $ satisfies the recurrence relation
	\begin{equation*}
		b_{n}(\lambda) = -\lambda b_{n-1}(\lambda)-\frac{1}{4}\sum_{i=2}^{n}(-2)^{i}b_{n-i}(\lambda).
	\end{equation*}
\end{corollary}

Corollary \ref{thm:50} follows from Theorem \ref{thm:13} with $ a_{-1}=1, \ a_0=0 $, and $ a_{i-1}=2^{i-2} $.

\begin{theorem}
	\label{thm:4}
	The solution of the recurrence relation
	\begin{equation*}
		b_{n}(\lambda) = -\lambda b_{n-1}(\lambda)-\frac{1}{4}\sum_{i=2}^{n}(-2)^{i}b_{n-i}(\lambda)
	\end{equation*}
	with initial condition $ b_{0}(\lambda) = 1 $ is
	\begin{equation*}
		b_{n}(\lambda) = \sum_{t=0}^{n}\sum_{k=0}^{n}\sum_{\ell=0}^{k}\binom{k}{t}\binom{t}{\ell}(-1)^{k}2^{2k-n+t-2\ell}\left[4\binom{k-t}{2k-n-\ell+1}+\binom{k-t}{2k-n-\ell}\right] \lambda^{t}.
	\end{equation*}
\end{theorem}

\bigskip

The proof of Theorem \ref{thm:4} is similar to the one of Theorem \ref{thm:41}, so it is deferred to Appendix \ref{B.2}. Finally, we give a formula for the entries of an eigenvector $ x $ associated to an eigenvalue $ \lambda $ of $ B_{n} $. The proof of the corollary follows from Theorem \ref{thm:14} with $ a_{-1} = 1 $.

\begin{corollary}
	Let $ x=(x_{n-1},x_{n-2},\ldots,x_{0})^{\top} $ be an eigenvector associated to an eigenvalue $ \lambda $ of $ B_{n} $ for $ n \geq 2 $. Then the entries of the vector $ x $ are of the form:
	
	\begin{equation*}
		x_{i} = (-1)^{i}b_{i}(\lambda)x_{0},
	\end{equation*}
	
	where $ b_{i}(\lambda) $ is the characteristic polynomial of $ B_{i} $.
\end{corollary}

\section{Conclusions}

In this work, we have obtained new production matrices for several classes of graphs. Also, eigenvectors and characteristic polynomials for each matrix, and several other formulas were derived from the production matrices. This work is based on previous results about enumeration of graphs based on production matrices. Our results have been possible due to different definitions of the degree we used with each matrix. While in previous papers \textcolor{blue}{\cite{Clemens2, Geometric}}, the degree of the graphs was defined as the number of incident edges to the root vertex of the graph or the number of vertices on each block in the case of partitions, in this work, we used two new definitions of degree: number of visible vertices seen from the new root vertex, and the number of isolated vertices seen from the new root vertex. Different definitions of degree lead to different production matrices, as we have seen in the cases of plane geometric graphs or non-crossing partitions.

However, there are other important properties of the matrices which remain open. For instance, a complete solution to the eigenvalues of the production matrices remains elusive to us. Only the assymptotic growth of the largest eigenvalue when $ n $ tends to infinity is known, because it is equal to the assymptotic growth of the number of graphs of the considered class.

An interesting open problem is to find production matrices for bipartite graphs with point sets in convex position. A major challenge is to obtain production matrices for point sets that are not in convex position.\\
\bigskip

{\small \noindent \textbf{Acknowledgments}}

\vspace{1mm}

\begin{wrapfigure}{l}{0.12\textwidth}
	\vspace{-20pt}
	\includegraphics[trim=10cm 6cm 10cm 5cm,clip,scale=0.15]{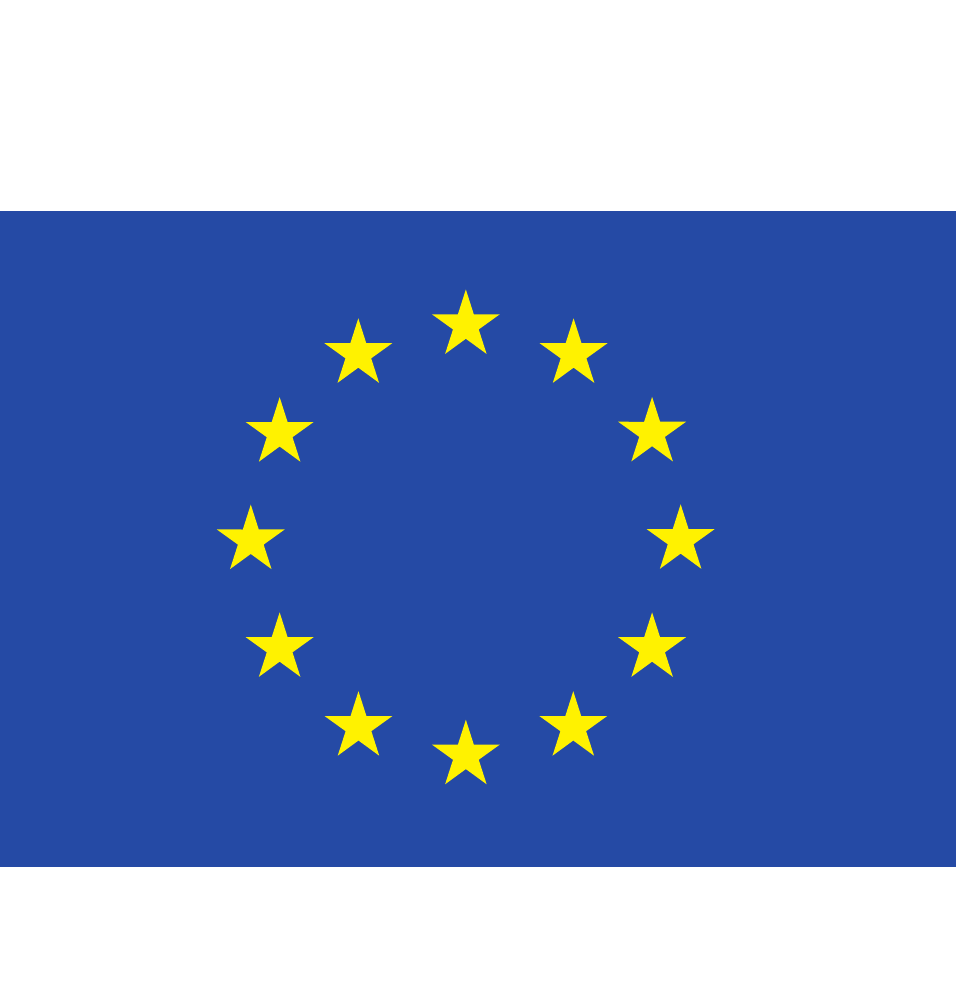}
	\vspace{-20pt}
\end{wrapfigure}

\small
\noindent This project has received funding from the European Union's Horizon 2020 research and innovation programme under the Marie Sk\l{}odowska-Curie grant agreement No 734922. C. H. and R. I. S. were supported by MINECO MTM2015-63791-R and Gen. Cat. 2017SGR1336 and 2017SGR1640.

\newpage

\begin{appendices}
	\addcontentsline{toc}{chapter}{Appendices}
	
	\section{Number of graphs}
	
		\subsection{Proof of Theorem \ref{thm:37}} \label{A.1}
		
			\begin{proof}
				The production matrix $ G_{n} $ has a corresponding Riordan Array. Using notation as in Section $ 2.1 $, the second row of $ G_{n} $ gives the A-sequence $ \{2,2,4,8,16,\ldots\} $ with generating function $ A(t) = \frac{2-2t}{1-2t} $ and its Z-sequence is given by the first row of $ G_{n} $ and is $ \{2,4,8,16,32,\ldots\} $ with generating function $ Z(t) = \frac{2}{1-2t} $. We have $ d(0) = 2 $ (because we use as starting vector $ v^{2} $), and by formula (\ref{eq:13})

				\begin{equation*}
					h(t) = A(th(t)) = \frac{2-2th(t)}{1-2th(t)}.
				\end{equation*}

				We calculate $ d(t) $ also with the formula given by (\ref{eq:13}):

				\begin{equation*}
					d(t) = \frac{d(0)}{1-tZ(th(t))} = \frac{2}{1-t(\frac{2}{1-2th(t)})} = \frac{2-4th(t)}{1-2th(t)-2t}.
				\end{equation*}

				We substitute $ th(t) = w $ and obtain
				{\fontsize{8.7}{7}\selectfont
				\begin{multline*}
					d(t) = \frac{2-4w}{1-2w-2\frac{w}{h(t)}} = \frac{h(t)(2-4w)}{h(t)-2wh(t)-2w} = \frac{\frac{2-2w}{1-2w}(2-4w)}{\frac{2-2w}{1-2w}-2w\frac{2-2w}{1-2w}-2w} = \frac{(2-2w)(2-4w)}{(2-2w)-2w(2-2w)-2w(1-2w)} = \\
					= \frac{(2-2w)(2-4w)}{(2-2w)(1-2w)-2w(1-2w)} = \frac{(2-2w)(2-4w)}{(1-2w)(2-4w)} = \frac{2-2w}{1-2w}.
				\end{multline*}}
				
				From this result, we see that $ d(t) = h(t) $. Then, Theorem \ref{thm:3} states that $ d_{n,j} = \frac{j+1}{n+1}[t^{n-j}]A(t)^{n+1} $. It follows from this theorem together with Theorem \ref{thm:5}, (and a shift of index) that the $ j $-th entry of $ v^{n} $ is the coefficient of $ t^{n-j-1} $ in the Taylor expansion of $ \frac{j}{n-1}(A(t))^{n-1} $. We expand

				\begin{equation*}
					A(t)^{n-1} = \left(\frac{2-2t}{1-2t}\right)^{n-1} = 2^{n-1}\left(\frac{1}{1-2t} - \frac{t}{1-2t}\right)^{n-1}.
				\end{equation*}

				And now, using the binomial formula for the algebraic expansion of powers of a binomial and using the well-known identity (\ref{eq:8}), we get the following result.
				\begin{multline*}	
					A(t)^{n-1} = 2^{n-1}\sum_{k=0}^{n-1} \binom{n-1}{k}\left(\frac{-t}{1-2t}\right)^{n-1-k}\left(\frac{1}{1-2t}\right)^{k} = 2^{n-1}\sum_{k=0}^{n-1} \binom{n-1}{k}(-t)^{n-1-k}\left(\frac{1}{1-2t}\right)^{n-1}	= \\
					= 2^{n-1}\sum_{k=0}^{n-1} \binom{n-1}{k}(-t)^{n-1-k}\left(\frac{1}{1-2t}\right)^{n-1} = 2^{n-1}\sum_{k=0}^{n-1} \binom{n-1}{k}(-1)^{n-1-k}t^{n-1-k}\sum_{\ell=0}^{\infty} \binom{n-1+\ell-1}{\ell} (2t)^{\ell} = \\
					= 2^{n-1}\sum_{k=0}^{n-1} \binom{n-1}{k}(-1)^{n-1-k}\sum_{\ell=0}^{\infty} \binom{n-1+\ell-1}{\ell} 2^{\ell}t^{n-1-k+\ell}.
				\end{multline*}
				It remains to determine the coefficient of $ t^{n-j-1} $. We set $ n-j-1 = n-1-k+\ell $ and therefore $ \ell = k - j $. Note that in the summation we can put $ k \geq j $. We thus arrive at the claimed formula
				\begin{multline*}
					v_{j}^{n} = \frac{j}{n-1}2^{n-1}\sum_{k=j}^{n-1}\binom{n-1}{k}(-1)^{n-1-k}\binom{n+k-j-2}{k-j}2^{k-j} = \\
					= \frac{j}{n-1}2^{n-1-j}\sum_{k=j}^{n-1}\binom{n-1}{k}(-1)^{n-1-k}\binom{n+k-j-2}{k-j}2^{k}. \tag*{\qedhere}
				\end{multline*} 
			\end{proof}
		
		\subsection{Proof of Theorem \ref{thm:7}} \label{A.2}
			\begin{proof}
				The production matrix $ C_{n} $ has a corresponding Riordan Array. Using notation as in Section $ 2.1 $, the second row of $ C_{n} $ gives the A-sequence $ \{1,3,7,15,31,\ldots\} $ with generating function $ A(t) = \frac{1}{(1-2t)(1-t)} $ and its Z-sequence is given by the first row of $ C_{n} $ and is $ \{3,7,15,31,63,\ldots\} $ with generating function $ Z(t) = \frac{3-2t}{(1-2t)(1-t)} $. We have $ d(0) = 1 $ and by formula (\ref{eq:13})
				\begin{equation*}
					h(t) = A(th(t)) = \frac{1}{(1-2th(t))(1-th(t))}.
				\end{equation*}

				We calculate $ d(t) $ using formula (\ref{eq:13})

				\begin{equation*}
					d(t) = \frac{d(0)}{1-tZ(th(t))} = \frac{1}{1-t\left(\frac{3-2th(t)}{(1-2th(t))(1-th(t))}\right)} = \frac{(1-2th(t))(1-th(t))}{(1-2th(t))(1-th(t))-3t+2t^{2}h(t)}.
				\end{equation*}

				We substitute $ th(t) = w $ and obtain
				\begin{multline*}
					d(t) = \frac{(1-2w)(1-w)}{(1-2w)(1-w)-3\frac{w}{h(t)}+2\frac{w^{2}}{h(t)}} = \frac{h(t)(1-2w)(1-w)}{h(t)(1-2w)(1-w)-3w+2w^{2}} = \\
					=  \frac{\frac{1}{(1-2w)(1-w)}(1-2w)(1-w)}{\frac{1}{(1-2w)(1-w)}(1-2w)(1-w)-3w+2w^{2}} = \frac{1}{1-3w+2w^2} = \frac{1}{(1-2w)(1-w)}.
				\end{multline*}

				From this result, we see that $ d(t) = h(t) $. Then, Theorem \ref{thm:3} states that $ d_{n,j} = \frac{j+1}{n+1}[t^{n-j}]A(t)^{n+1} $. It follows from this theorem together with Theorem \ref{thm:5}, (and a shift of index) that the $ j $-th entry of $ v_{n} $ is the coefficient of $ t^{n-j-1} $ in the Taylor expansion of $ \frac{j}{n-1}(A(t))^{n-1} $. Firstly, we use the partial fraction decomposition, and divide the generating function $ A(t) $ into two terms in order to apply the binomial theorem to them.
				\begin{multline*}
					A(t)^{n-1} = \left(\frac{1}{(1-2t)(1-t)}\right)^{n-1} = \left(\frac{2}{1-2t} - \frac{1}{1-t}\right)^{n-1} = \\
					= \sum_{k=0}^{n-1} \binom{n-1}{k}\left(\frac{2}{1-2t}\right)^{n-1-k}\left(\frac{-1}{1-t}\right)^{k} = 2^{n-1}\sum_{k=0}^{n-1} \binom{n-1}{k}2^{-k}\left(\frac{1}{1-2t}\right)^{n-1-k}(-1)^{k}\left(\frac{1}{1-t}\right)^{k}.
				\end{multline*}
				
				Secondly, we apply the identity given by (\ref{eq:8}), so the previous equation is equal to
				\begin{multline*}
					A(t)^{n-1} = 2^{n-1}\sum_{k=0}^{n-1} \binom{n-1}{k}\left(-\frac{1}{2}\right)^{k}\sum_{\ell=0}^{\infty} \binom{n-k+\ell-2}{\ell} (2t)^{\ell}\sum_{m=0}^{\infty}\binom{k+m-1}{m}t^{m} = \\
					= 2^{n-1}\sum_{k=0}^{n-1} \binom{n-1}{k}\left(-\frac{1}{2}\right)^{k}\sum_{\ell=0}^{\infty} \binom{n-k+\ell-2}{\ell} 2^{\ell}\sum_{m=0}^{\infty}\binom{k+m-1}{m}t^{\ell+m}.
				\end{multline*}

				It remains to determine the coefficient of $ t^{n-j-1} $. We set $ n-j-1 = \ell+m $ and therefore $ m = n-\ell-j-1 $. 

				\begin{equation*}
					A(t)^{n-1} = 2^{n-1}\sum_{k=0}^{n-1} \binom{n-1}{k}\left(-\frac{1}{2}\right)^{k}\sum_{\ell=0}^{\infty} \binom{n-k+\ell-2}{\ell} 2^{\ell}\binom{k+n-\ell-j-2}{n-\ell-j-1}t^{n-j-1}.
				\end{equation*}

				Note that in the summation we can put $ \ell \leq n-j-1 $. We thus arrive at the claimed formula:

				\begin{equation*}
					v_{j}^{n} = \frac{j}{n-1}2^{n-1}\sum_{k=0}^{n-1}\binom{n-1}{k}\left(-\frac{1}{2}\right)^{k}\sum_{\ell= 0}^{n-j-1}\binom{n-2-k+\ell}{\ell}\binom{k+n-\ell-j-2}{n-\ell-j-1}2^{\ell}. \tag*{\qedhere}
				\end{equation*}
			\end{proof}
		
		\subsection{Proof of Theorem \ref{thm:2}}\label{A.3}
			\begin{proof}
				The production matrix $ B_{n} $ has a corresponding Riordan Array. Using notation as in Section $ 2.1 $, the second row of $ B_{n} $ gives the A-sequence $ \{1,0,1,2,4,\ldots\} $ with generating function $ A(t) = \frac{(t-1)^2}{1-2t} $ and its Z-sequence is given by the first row of $ B_{n} $ and is $ \{0,1,2,4,8,\ldots\} $ with generating function $ Z(t) = \frac{t}{1-2t} $. We have $ d(0) = 1 $ and by formula (\ref{eq:13})
				\begin{equation*}
					h(t) = A(th(t)) = \frac{(th(t)-1)^2}{1-2th(t)}.
				\end{equation*}

				We calculate $ d(t) $:

				\begin{equation*}
					d(t) = \frac{d(0)}{1-tZ(th(t))} = \frac{1}{1-t(\frac{th(t)}{1-2th(t)})} = \frac{1-2th(t)}{1-2th(t)-t^2h(t)}.
				\end{equation*}

				We substitute $ th(t) = w $ and obtain
				{\fontsize{8.63}{7}\selectfont
				\begin{multline*}
					d(t) = \frac{1-2w}{1-2w-\frac{w^2}{h(t)}} = \frac{h(t)(1-2w)}{h(t)-2wh(t)-w^2} = \frac{\frac{(w-1)^2}{1-2w}(1-2w)}{\frac{(w-1)^2}{1-2w}-\frac{2w(w-1)^2}{1-2w}-w^2} = \frac{(w-1)^2(1-2w)}{(w-1)^2-2w(w-1)^2-w^2(1-2w)} = \\ = \frac{(w-1)^2(1-2w)}{(w-1)^2(1-2w)-w^2(1-2w)} = \frac{(w-1)^2(1-2w)}{[(w-1)^2-w^2](1-2w)} = \frac{(w-1)^2}{1-2w}.
				\end{multline*}}
				
				From this result, we see that $ d(t) = h(t) $. Then, Theorem \ref{thm:3} states that $ d_{n,j} = \frac{j+1}{n+1}[t^{n-j}]A(t)^{n+1} $. It follows from this theorem together with Theorem \ref{thm:5}, (and a shift of index) that the $ j $-th entry of $ v^{n} $ is the coefficient of $ t^{n-j+1} $ in the Taylor expansion of $ \frac{j}{n+1}(A(t))^{n+1} $. Using the well-known identity (\ref{eq:8}) we expand
				\begin{multline*}
					A(t)^{n+1} = \left(\frac{(t-1)^2}{1-2t}\right)^{n+1} = \left(\frac{t^2}{1-2t} + 1\right)^{n+1} = \sum_{k=0}^{n+1} \binom{n+1}{k}\left(\frac{t^2}{1-2t}\right)^{n+1-k} = \\
					= \sum_{k=0}^{n+1} \binom{n+1}{k}t^{2n+2-2k} \left(\frac{1}{1-2t}\right)^{n+1-k}
					= \sum_{k=0}^{n+1} \binom{n+1}{k}t^{2n+2-2k}\sum_{\ell=0}^{\infty} \binom{n+1-k+\ell-1}{\ell} (2t)^{\ell} = \\
					= \sum_{k=0}^{n+1}\sum_{\ell=0}^{\infty}\binom{n+1}{k}\binom{n-k+\ell}{\ell}2^{\ell}t^{\ell+2n+2-2k}.
				\end{multline*}

				It remains to determine the coefficient of $ t^{n-j+1} $. We set $ n-j+1 = \ell+2n+2-2k $ and therefore $ \ell = 2k-n-j-1 $. Note that $ k \geq \frac{n+j+1}{2} $. We thus arrive at the claimed formula

				\begin{equation*}
					v_{j}^{n} = \frac{j}{n+1}\sum_{k=\big\lceil\frac{n+j+1}{2}\big\rceil}^{n+1}\binom{n+1}{k}\binom{k-j-1}{2k-n-j-1}2^{2k-n-j-1}.\tag*{\qedhere}
				\end{equation*}
			\end{proof}

	\section{Characteristic polynomials}
	
		\subsection{Proof of Theorem \ref{thm:8}}\label{B.1}
			\begin{proof}
				Consider the infinite matrix
				
				\begin{equation*}
					M = \begin{pmatrix}
						3-\lambda & -(2^3-1) & 2^4-1 & -(2^5-1) & \cdots \\
						1 & 0 & 0 & 0 & \cdots \\
						0 & 1 & 0 & 0 & \cdots \\
						0 & 0 & 1 & 0 & \cdots \\
						0 & 0 & 0 & 1 & \cdots \\
						\vdots & \vdots & \vdots & \vdots & \ddots \\
					\end{pmatrix}
				\end{equation*}
				
				and let $ w_{0} $ be the vector $ (1,0,\ldots)^{\top} $, and let $ w_{i} $ be the vector whose first $ i $ entries are the first $ i $ characteristic polynomials $ c_{i}(\lambda) $, and the remaining ones are zero. That is, $ w_{i} = (c_{i}(\lambda), c_{i-1}(\lambda), \ldots, c_{1}(\lambda), 0, \ldots)^{\top} $. Then $ M \cdot w_{i} = w_{i+1} $. It follows that $ w_{n} $ is the first column of $ M^{n} $. We can now use the Riordan Array approach. The Z-sequence is $ \{3-\lambda, -7, 15, \ldots\} $ with generating function $ Z(t) = -\lambda + \frac{4}{1+2t}- \frac{1}{1+t}$. The A-sequence is $ \{1,0,\ldots\} $ with generating function $ A(t) = 1 $. It follows that $ h(t) = 1 $ and
				
				\begin{equation*}
					d(t) = \frac{1}{1-t(-\lambda + \frac{4}{1+2t} - \frac{1}{1+t})} = \frac{1+3t+2t^{2}}{1+t(\lambda(1+3t+2t^{2}))} = d_{1}(t)+3td_{1}(t) + 2t^{2}d_{1}(t),
				\end{equation*}
				
				where
				\begin{equation*}
					d_{1}(t) = \frac{1}{1+t(\lambda(1+3t+2t^{2}))}.
				\end{equation*}
				
				Then,
				\begin{equation}
					\label{eq:1}
					d_{n,j} = [t^n]d(t)(th(t))^{j} = [t^n]d(t)t^{j} = [t^{n-j}]d(t) = [t^{n-j}]d_{1}(t) + 3[t^{n-1-j}]d_{1}(t) + 2[t^{n-2-j}]d_{1}(t).
				\end{equation}
				
				We set $ z= \lambda(1+3t+2t^{2}) $, and using the well-known identity (\ref{eq:5}):
				\begin{equation*}
					d_{1}(t)=\sum_{k=0}^{\infty}(-1)^kt^{k}(\lambda(1+3t+2t^{2}))^{k}.
				\end{equation*}

				We apply the binomial theorem to $ (\lambda(1+3t+2t^{2}))^{k} $. We then only have to take the coefficient of $ t^{n-j} $ in $ d_{1}(t) $ for the first term of (\ref{eq:1}), the coefficient of $ t^{n-1-j} $ in $ d_{1}(t) $ for the second term of (\ref{eq:1}) and the coefficient of $ t^{n-2-j} $ for the third term of (\ref{eq:1}). After some short calculations, this gives the formula:
				{\fontsize{8.4}{7}\selectfont
				\begin{multline*}
					c_{n}(\lambda) = \sum_{k=0}^{n}\sum_{\ell=0}^{k}(-1)^{k}\binom{k}{\ell}\binom{\ell}{n+2\ell-3k}2^{k-\ell}3^{n+2\ell-3k}\lambda^{k} + \sum_{k=0}^{n}\sum_{\ell=0}^{k}(-1)^{k}\binom{k}{\ell}\binom{\ell}{n+2\ell-3k-1}2^{k-\ell}3^{n+2\ell-3k-1} \lambda^{k} \\
					+ \sum_{k=0}^{n}\sum_{\ell=0}^{k}(-1)^{k}\binom{k}{\ell}\binom{\ell}{n+2\ell-3k-2}2^{k-\ell+1}3^{n+2\ell-3k-2} \lambda^{k}.
				\end{multline*}}
				We finally rewrite $ c_{n}(\lambda) = \sum_{t=0}^{n} c_{t}\lambda^{t} $, with $ c_{t} $ the coefficient of $ \lambda^{t} $. Since $ \lambda^{t} = \lambda^{k} $ we set $ k = t $ and get after some elementary operations:

				\begin{equation*}
					c_{n}(\lambda) = \sum_{t=0}^{n}\left(\sum_{\ell=0}^{t}\binom{t}{\ell}(-1)^{t}2^{t-\ell}3^{n+2\ell-3t-2}\left[2\binom{\ell}{n+2\ell-3t-2}+9\binom{\ell+1}{n+2\ell-3t}\right]\right)\lambda^{t}. \tag*{\qedhere}
				\end{equation*}
			\end{proof}
			
		\subsection{Proof of Theorem \ref{thm:4}}\label{B.2}
			\begin{proof}
				Consider the infinite matrix
				
				\begin{equation*}
					M = \begin{pmatrix}
						-\lambda & -2^0 & 2^1 & -2^2 & \cdots \\
						1 & 0 & 0 & 0 & \cdots \\
						0 & 1 & 0 & 0 & \cdots \\
						0 & 0 & 1 & 0 & \cdots \\
						0 & 0 & 0 & 1 & \cdots \\
						\vdots & \vdots & \vdots & \vdots & \ddots \\
					\end{pmatrix}
				\end{equation*}
				
				and let $ w_{0} $ be the vector $ (1,0,\ldots)^{\top} $, and let $ w_{i} $ be the vector whose first $ i $ entries are the first $ i $ characteristic polynomials $ b_{i}(\lambda) $, and the remaining ones are zero. That is, $ w_{i} = (b_{i}(\lambda), b_{i-1}(\lambda), \ldots, b_{1}(\lambda), 0, \ldots) $. Then $ M \cdot w_{i} = w_{i+1} $. It follows that $ w_{n} $ is the first column of $ M^{n} $. We can now use the Riordan Array approach. The Z-sequence is $ \{-\lambda, -1, 2, -4, \ldots\} $ with  generating function $ Z(t) = -\lambda - \frac{t}{1+2t}$. The A-sequence is $ \{1,0,\ldots\} $ with generating function $ A(t) = 1 $. It follows that $ h(t) = 1 $ and
				
				\begin{equation*}
					d(t) = \frac{1}{1-t(-\lambda - \frac{t}{1+2t})} = \frac{1+2t}{1+t((2+\lambda)+(2\lambda+1)t)}.
				\end{equation*}
				
				Then,
				\begin{equation}
					\label{eq:2}
					d_{n,j} = [t^n]d(t)(th(t))^{j} = [t^n]d(t)t^{j} = [t^{n-j}]d(t) = [t^{n-j}]d_{1}(t) + 2[t^{n-1-j}]d_{1}(t),
				\end{equation}
				
				where
				\begin{equation*}
					d_{1}(t) = \frac{1}{1+t((2+\lambda)+(2\lambda+1)t)}.
				\end{equation*}
				
				We set $ z= (2+\lambda)+(2\lambda+1)t $, and using equation (\ref{eq:5}):

				\begin{equation*}
					d_{1}(t)=\sum_{k=0}^{\infty}(-1)^kt^{k}((2+\lambda)+(2\lambda+1)t)^{k}.
				\end{equation*}

				We apply the binomial theorem to $ ((2+\lambda)+(2\lambda+1)t)^{k} $. We then only have to take the coefficient of $ t^{n-j} $ in $ d_{1}(t) $ for the left term of (\ref{eq:2}) and the coefficient of $ t^{n-1-j} $ in $ d_{1}(t) $ for the right term of (\ref{eq:2}). After some short calculations, this gives the formula
				\begin{multline*}
					b_{n}(\lambda) = \sum_{k=0}^{n}\sum_{\ell=0}^{2k-n}\sum_{i=0}^{n-k}(-1)^{k} \binom{k}{n-k}\binom{2k-n}{\ell}\binom{n-k}{i} 2^{2k-n-\ell+i} \lambda^{\ell+i}+\\
					+2\sum_{k=0}^{n}\sum_{\ell=0}^{2k-n+1}\sum_{i=0}^{n-k-1}(-1)^{k}\binom{k}{n-k-1}\binom{2k-n+1}{\ell}\binom{n-k-1}{i} 2^{2k-n-\ell+i+1} \lambda^{\ell+i}.
				\end{multline*}

				Note that in the summation we can put $ \ell \leq k $.
				\begin{multline*}
					b_{n}(\lambda) = \sum_{k=0}^{n}\sum_{\ell=0}^{k}\sum_{i=0}^{n-k}(-1)^{k} \binom{k}{n-k}\binom{2k-n}{\ell}\binom{n-k}{i} 2^{2k-n-\ell+i} \lambda^{\ell+i}+\\
					+2\sum_{k=0}^{n}\sum_{\ell=0}^{k}\sum_{i=0}^{n-k-1}(-1)^{k}\binom{k}{n-k-1}\binom{2k-n+1}{\ell}\binom{n-k-1}{i} 2^{2k-n-\ell+i+1} \lambda^{\ell+i}.
				\end{multline*}

				We finally rewrite $ b_{n}(\lambda) = \sum_{t=0}^{n} c_{t}(\lambda)\lambda^{t} $, with $ c_{t}(\lambda) $ the coefficient of $ \lambda^{t} $. Since $ \lambda^{t} = \lambda^{\ell+i} $ we set $ i = t-\ell $ and get
				\begin{multline*}
					b_{n}(\lambda) = \sum_{t=0}^{n}\sum_{k=0}^{n}\sum_{\ell=0}^{k}(-1)^{k} \binom{k}{n-k}\binom{2k-n}{\ell}\binom{n-k}{t-\ell} 2^{2k-n+t-2\ell} \lambda^{t}+\\
					+\sum_{t=0}^{n}\sum_{k=0}^{n}\sum_{\ell=0}^{k}(-1)^{k}\binom{k}{n-k-1}\binom{2k-n+1}{\ell}\binom{n-k-1}{t-\ell} 2^{2k-2\ell-n+t+2} \lambda^{t}.
				\end{multline*}

				After some elementary operations, and taking into account the subset of a subset identity $ \binom{n}{m}\binom{m}{k} = \binom{n}{k}\binom{n-k}{m-k} $ we finally get the formula

				\begin{equation*}
					b_{n}(\lambda) = \sum_{t=0}^{n}\sum_{k=0}^{n}\sum_{\ell=0}^{k}\binom{k}{t}\binom{t}{\ell}(-1)^{k}2^{2k-n+t-2\ell}\left[4\binom{k-t}{2k-n-\ell+1}+\binom{k-t}{2k-n-\ell}\right] \lambda^{t}. \tag*{\qedhere}
				\end{equation*}
			\end{proof}
	
\end{appendices}

\end{document}